\documentclass[12pt]{amsart}

\usepackage{amsfonts,amssymb,amsmath,amscd,amstext}
\usepackage[colorlinks=true,linkcolor=blue,citecolor=blue]{hyperref}
\usepackage[utf8]{inputenc}
\usepackage{microtype}
\usepackage{graphicx}
\usepackage{changes}
\usepackage{comment}
\usepackage{mathtools}
\usepackage[a4paper,top=2.5cm,bottom=2.5cm,left=2.5cm,right=2.5cm]{geometry}

\usepackage{cancel}

\renewcommand{\leq}{\leqslant}
\renewcommand{\geq}{\geqslant}
\renewcommand{\le}{\leqslant}
\renewcommand{\ge}{\geqslant}

\newcommand\restr[2]{{
  \left.\kern-\nulldelimiterspace 
  #1 
  \vphantom{\big|} 
  \right|_{#2} 
  }}

\newcommand{\G}{{\mathbb{G}}}

\newcommand{\Om}{\Omega}

\newcommand{\B}{{\mathbb{B}}}
\newcommand{\E}{{\mathbb{E}}}

\newcommand{\average}{{\mathchoice {\kern1ex\vcenter{\hrule height.4pt
				width 6pt
				depth0pt} \kern-9.7pt} {\kern1ex\vcenter{\hrule height.4pt width 4.3pt
				depth0pt}
			\kern-7pt} {} {} }}

\newcommand{\mres}{\mathbin{\vrule height 1.6ex depth 0pt width
0.13ex\vrule height 0.13ex depth 0pt width 1.3ex}}

\newtheorem{theorem}{Theorem}[section]
\newtheorem{proposition}[theorem]{Proposition}
\newtheorem{lemma}[theorem]{Lemma}
\newtheorem{corollary}[theorem]{Corollary}

\theoremstyle{definition}

\newtheorem{definition}[theorem]{Definition} 

\theoremstyle{remark}

\numberwithin{equation}{section}

\newcommand{\R}{\mathbb{R}}

\newcommand{\cE}{\mathcal{E}}

\newcommand{\cH}{\mathcal{H}}

\newcommand{\sm}{\setminus}

\renewcommand{\exp}{\mbox{\rm exp}\;\!}

\newcommand{\spn}{\mbox{\rm span}}

\newcommand{\Lie}{\mathrm{Lie}}

\newcommand{\m}{\mathrm m}

\newcommand{\NN}{\mathrm N}

\newcommand{\der}{\partial}

\newcommand{\beqas}{\begin{eqnarray*}}
\newcommand{\eeqas}{\end{eqnarray*}}
\newcommand{\beqa}{\begin{eqnarray}}
\newcommand{\eeqa}{\end{eqnarray}}
\newcommand{\beq}{\begin{equation}}
\newcommand{\eeq}{\end{equation}}
\newcommand{\bce}{\begin{center}}
\newcommand{\ece}{\end{center}}

\newcommand{\set}[1]{\left\{ #1 \right\}}            

\newtheorem{The}{Theorem}[section]
\newtheorem{Lem}[The]{Lemma}
\newtheorem{Def}[The]{Definition}
\newtheorem{Rem}[The]{Remark}
\newtheorem{Pro}[The]{Proposition}
\newtheorem{Cor}[The]{Corollary}
\newtheorem{Exa}[The]{Example}
\newtheorem{Con}{Conjecure}

\newcommand{\bt}{\begin{The}}
\newcommand{\et}{\end{The}}
\newcommand{\bl}{\begin{Lem}}
\newcommand{\el}{\end{Lem}}
\newcommand{\bd}{\begin{Def}\rm}
\newcommand{\ed}{\end{Def}}
\newcommand{\br}{\begin{Rem}\rm}
\newcommand{\er}{\end{Rem}}
\newcommand{\bpr}{\begin{Pro}}
\newcommand{\epr}{\end{Pro}}
\newcommand{\bc}{\begin{Cor}}
\newcommand{\ec}{\end{Cor}}
\newcommand{\bj}{\begin{Con}}
\newcommand{\ej}{\end{Con}}
\newcommand{\bex}{\begin{Exa}}
\newcommand{\eex}{\end{Exa}}

\begin{document}

\title[A minimal regularity for the area formula in the Engel group]{A minimal regularity for the area formula \\
in the Engel group}
\author[F. Corni]{Francesca Corni}
\address[Francesca Corni]{ 
Dipartimento di Matematica, 
Universit\`a di Bologna, 
Piazza di Porta S.Donato 5, 40126 Bologna, Italy }
\email[Francesca Corni]{francesca.corni3@unibo.it}
\author[F. Essebei]{Fares Essebei}
\address[Fares Essebei]{ 
Dipartimento di Matematica, 
Universit\`a di Pisa, 
Largo Bruno Pontecorvo 5, 56127 Pisa, Italy }
\email[Fares Essebei]{fares.essebei@dm.unipi.it}
\author[V. Magnani]{Valentino Magnani}
\address[Valentino Magnani]{ 
Dipartimento di Matematica, 
Universit\`a di Pisa, 
Largo Bruno Pontecorvo 5, 56127 Pisa, Italy }
\email[Valentino Magnani]{valentino.magnani@unipi.it}

\begin{abstract}
We prove that the upper blow-up theorem in the Engel group holds for $C^1$ submanifolds. Combining this result with the known negligibility of the singular set, we obtain an integral representation of the spherical measure for all surfaces of class $C^{1,\alpha}$ in the Engel group. A new and central aspect of our method is the suitable application of Stokes' theorem to prove the upper blow-up. The proof also relies on the simple algebraic structure of the Maurer-Cartan equations in the Engel group. Some general tools are also introduced to establish the area formula in an arbitrary Carnot group.
\end{abstract}

\thanks{F.C. is partially supported by INdAM-GNAMPA 2024 project CUP E53C23001670001: Free boundary problems in non commutative structures and degenerate operators. V.M. is partially supported by the APRISE - {\em Analysis and Probability in Science} project, funded by the University of Pisa, grant PRA 2022 85, by PRIN 2022PJ9EFL {\em Geometric Measure Theory: Structure of Singular Measures, Regularity Theory and Applications in the Calculus of Variations}, funded by the European Union--NextGenerationEU, CUP E53D23005860006, and by the MIUR Excellence Department Project awarded to the Department of Mathematics, University of Pisa, CUP I57G22000700001.}



\maketitle
\tableofcontents

\section{Introduction}

The study of the area formula for submanifolds in Carnot groups has significantly developed over the past two decades, and it is still an active area of research, lying at the intersection of analysis, differential geometry, and geometric measure theory.
A number of motivations stem from Gromov’s program, laid out in the seminal 1996 paper \cite{Gromov1996}, which also provides a general formula for the Hausdorff dimension of smooth submanifolds with respect to the Carnot–Carathéodory distance.

The appearance of the Hausdorff measure with respect to the Carnot-Carathéodory distance 
took place before, in the first works by Pansu, concerning the isoperimetric inequality in the Heisenberg group, \cite{Pansu82,Pansu82ineq}, and 
in Heinonen's notes  \cite{Hei1995CalcCG}. More recent results on the area formula for the spherical measure of submanifolds can be found for instance in \cite{Magnani2019Area,CorMag23pr,JNGV22,AntMer22,Vit22,CorMag25}, but this list could be enlarged. 

We point out that some tools of the present work actually hold for a general homogeneous group, where the stratification of the Lie algebra is not necessary, see both Sections~\ref{sect:basic} and Section~\ref{sect:genres}. For the area formula we focus our attention on the Engel group $\E$, that is a specific Carnot group of step three. The peculiar geometry of this group is well known in the literature, as it displays additional difficulties, for instance with respect to step-2 Carnot groups, and it becomes a natural test for general settings. 
The Engel group is a connected and simply connected Lie group, whose Lie algebra 
	$$ {\rm Lie}(\mathbb{E}) = V_1 \oplus V_2 \oplus V_3$$
has a basis $(X_1, X_2, X_3, X_4)$ with the only nontrivial relations
	\begin{equation}\label{particolarEngel}
		[X_1, X_2] = X_3 \quad \mbox{and} \quad [X_1, X_3] = X_4.
	\end{equation}
The strata $V_i$ of the Lie algebra are given by 
$$
V_1 = {\rm span}\{X_1, X_2\},\quad V_2 = {\rm span}\{X_3\}\quad \text{and}\quad V_3 = {\rm span}\{X_4\}.
$$
As in the Euclidean theory of sets of finite perimeter, also in Carnot groups the blow-up remains a central step to find a formula for the spherical measure. 
On the other hand, the anisotropy of dilations, the different degrees of points of a submanifold, and the possible complexity of the Lie algebra of the underlying group make the blow-up in arbitrary Carnot groups a largely open question. Moreover, with respect to the Euclidean theory, the problem of establishing the negligibility of {\em singular points} also appears. These points are characterized by having lower degree, with respect to the degree of the submanifold.

The pointwise degree of a point $p\in\Sigma$ in a submanifold $\Sigma\subset\G$ is a suitable positive 
integer $d_\Sigma(p)$ associated to the tangent space $T_p\Sigma$. Broadly speaking, it is a sort of ``formal pointwise dimension", depending on the intersection of the strata of $\Lie(\G)$ with $T_p\Sigma$. The maximum of $d_\Sigma(\cdot)$ on $\Sigma$ defines the degree $d(\Sigma)$ of $\Sigma$, 
see Section~\ref{sect3.2degree} for precise definitions. It was previously introduced by Gromov in more geometric terms, \cite[Section 0.6.B]{Gromov1996}.

In \cite{LeDMag2010} a specific coordinate system was used for the Engel group in order to prove the $\cH^{d(\Sigma)}$-negligibility of the singular set 
\[
C(\Sigma)=\set{p\in\Sigma: d_\Sigma(p)<d(\Sigma)},
\]
when $\Sigma$ is of class $C^{1,\alpha}$, \cite[Remark~2]{LeDMag2010}.
Although we rely on this negligibility result, we include in Section~\ref{appendix} a self-contained proof, both to keep the exposition uniform by our main coordinate system and for the reader’s convenience. 

The blow-up at points $p\in\Sigma$ of maximum degree, i.e. $d_\Sigma(p)=d(\Sigma)$, needs the general coordinate system obtained by Theorem~\ref{t:specialcoord}. Using these coordinates, we introduce a useful tool for obtaining the blow-up in arbitrary homogeneous groups, that is Theorem~\ref{genres}. Its proof is based on the observation that the argument in \cite[Theorem~1.2]{Magnani2019Area} ultimately relies only on an appropriate intrinsic Taylor expansion of the special parametrization.
 
On the other hand, the general graded basis of $\Lie(\G)$ provided by Theorem~\ref{t:specialcoord} yields arbitrary structure coefficients $\xi_{ij}$, see  \eqref{structurecoeff}. In fact, in Section~\ref{generallambda} we need to work with a $\xi_{ij}$-parametrized coordinate system.
We point out that the area formula for $C^1$ submanifolds of codimension three and one in the Engel group follow from the general results of \cite{Magnani2019Area}. Then we focus our attention on 2-dimensional surfaces and their points of maximum degree. Furthermore, due to the nonintegrability of the horizontal distribution of $\E$ surfaces of degree two cannot exist. We provide a rigorous proof of this nonexistence following the approach of \cite{Mag2010pams}, see Section~\ref{sect:nonexistence}. 
Surfaces of degree five are transversal, so their area formula is also known from \cite{Magnani2019Area}. 

The only cases where the blow-up with $C^1$ smooth regularity were not known correspond to surfaces having degree three or four. 
The delicate case concerns degree-3 surfaces, where the upper blow-up with general coefficients $\xi_{ij}$ does not work if we are unable to detect further conditions on the structure coefficients. Such a technical obstacle was rather unexpected and it represents the new phenomenon of the present work. 

The precise condition is the vanishing of the structure coefficient $\xi_{13}$. 
We notice that this is also a coefficient of the Maurer-Cartan equation \eqref{eq:dtheta_4form}. Then a suitable application of Stokes’ theorem, along with the natural condition $d(\partial \Sigma)<d(\Sigma)$ (Lemma~\ref{gradosulbordo}), yields $\xi_{13}=0$ (Theorem~\ref{mainstokestheorem}), hence proving Proposition~\ref{Gammas}. Due to this result we can keep the $C^1$ regularity also for the blow-up of degree-3 surfaces, that is the lowest available regularity. However, as already mentioned, the negligibility of singular points is also needed, see Theorem~\ref{teotrascu}, and this will increase the regularity to obtain the area formula.
We have finally arrived at our main result.
 
\begin{theorem}\label{teo:areaengel}
	Let $\Sigma \subset \E$ be a $2$-dimensional $C^{1,\alpha}$ regular submanifold with $d(\Sigma)\le 4$. Then the following formula holds for every Borel set $B \subset \Sigma$  
	\begin{equation}\label{areafo}
		\mu_\Sigma(B) = \int_B \lvert \tau_{\Sigma, N}(p)\rvert_g d\sigma_g(p) = \int_B \beta_d(A_p \Sigma) \ d \mathcal{S}^N (p),
	\end{equation}
whenever $(d(\Sigma)-2)/d(\Sigma)<\alpha\le 1$.
\end{theorem}
The intrinsic measure $\mu_\Sigma$ is introduced in \eqref{intrinsicmeas} and $\beta_d(A_p \Sigma)$ is the \textit{spherical factor}, see Definition~\ref{sfe}. 
Combining Proposition \ref{Gammas}, Proposition \ref{Gammas2}, Theorem \ref{genres}, Theorem~\ref{teo:areageneral}, Theorem~\ref{teotrascu}, and Theorem~\ref{horizontalnon} the proof of Theorem~\ref{teo:areaengel} follows.
We point out that Theorem~\ref{teo:areageneral} could also be replaced by \cite[Theorem~1.2]{CorMag25}.
Since our results are focused on surfaces of degree three and four, in Section~\ref{sect:examples_surfaces} we construct some nontrivial examples of such surfaces. Other examples can be found in \cite{Mag13Vit} with respect to a different coordinate system.
If we equip the Engel group with a multiradial distance,
see \cite[Definition 1.3]{CorMag25}, we can apply
\cite[Theorem 1.5]{CorMag25}, which combined with Theorem \ref{teo:areaengel} yields the following.
\begin{corollary}
Let $\mathbb{E}$ be the Engel group equipped with a multiradial distance $d$. Let $\Sigma \subset \mathbb{E}$ be a $2$-dimensional $C^{1,\alpha}$ regular submanifold with $d(\Sigma)=N\le 4$. Thus, setting $\mathcal{S}_d^N=\omega_d(\mathcal{F}_{1,1,0}) \mathcal{S}^N$ if $N=3$ and $\mathcal{S}_d^N=\omega_d(\mathcal{F}_{1,0,1}) \mathcal{S}^N$ if $N=4$ according to \cite[Theorem~3.3]{CorMag25}, for every Borel set $B \subset \Sigma $ we have
\begin{equation}\label{areafoSym}
    \mathcal{S}_d^N(B) = \int_B \lvert \tau_{\Sigma, N}(p)\rvert_g d\sigma_g(p)
\end{equation}
whenever $(d(\Sigma)-2)/d(\Sigma)<\alpha\le 1$.
\end{corollary}
Useful examples of multiradial distances are the distance $d_{\infty}$, \cite{FSSC5}, and the family of homogeneous distances arising from \cite[Theorem 2]{HebSik90}.
We close this introduction by stating another corollary of Theorem~\ref{teo:areaengel}, that is more manageable for applications and relies on the area formula for 
transversal submanifolds and $C^1$ curves, \cite[Theorem~1.3]{Magnani2019Area}. Indeed surfaces of degree five and hypersurfaces in $\E$ are precisely transversal submanifolds. The same corollary also lowers the regularity required for the area formula to hold for all submanifolds in the Engel group, \cite{LeDMag2010}.

\begin{corollary}\label{cor:areaengel}
For every submanifold $\Sigma \subset \E$  of class $C^{1,\alpha}$, $\alpha>1/2$, and degree $\NN$, 
the following area formula holds   
	\begin{equation}\label{areafoReg}
		\mu_\Sigma(B) = \int_B \lvert \tau_{\Sigma, N}(p)\rvert_g d\sigma_g(p) = \int_B \beta_d(A_p \Sigma) \ d \mathcal{S}^N (p)
	\end{equation}
for every Borel set $B \subset \Sigma$.
\end{corollary}
\section{Preliminaries}

We divide into three parts all the notions that are necessary to study the area formula for the spherical measure in Carnot groups.

\subsection{Basic facts on homogeneous groups}\label{sect:basic}

A graded group $\G$ is a connected, simply connected and nilpotent Lie group, whose Lie algebra $\Lie(\G)$ is {\em graded}. Precisely, we have 
\[
\Lie(\G) = V_1 \oplus \ldots \oplus V_\iota \quad \text{and}\quad  [V_i, V_j] \subset V_{i+j} \quad \mbox{for all } i,j \ge 1,
\]
where $[V_i, V_j] =\spn\{ [X, Y] \ : \ X \in V_i, \, Y \in V_j\}$ and $V_i=\set{0}$ for $i>\iota$. The positive integer $\iota$ is called the \emph{step} of $\G$. We say that $\G$ is a \emph{stratified group} if we add the stronger condition that 
$[V_1, V_i] = V_{i+1}$ for $i=1,\ldots,\iota - 1$.

In our assumptions, the exponential map
$\exp : {\rm Lie}(\G) \rightarrow \G$ is a bianalytic diffeomorphism, that allows us to identify in a standard way $\G$ with ${\rm Lie}(\G)$. 
As a consequence, {\em we can identify $\G$ with a graded vector space} $H_1 \oplus \ldots \oplus H_\iota$, endowed with a Lie algebra structure, whose Lie group operation is given by the 
{\em Baker--Campbell--Hausdorff} formula, in short BCH formula. Throughout the paper we will always assume this condition on $\G=H_1\oplus\cdots\oplus H_\iota$, and we will say that $\G$ has {\em graded decomposition} $H_1\oplus\cdots\oplus H_\iota$. 

The terms $c_j$ of the BCH formula can be written by an explicit recursive formula, see for instance \cite[Lemma~2.15.3]{Var84Lie}, getting
$$
x y = \sum_{j=1}^\iota c_j(x,y) = x + y + \frac{[x,y]}{2} + \sum_{j=3}^\iota c_j(x,y)
$$
for all $x,y \in \G$.
Associated with the grading are the {\em dilations} $\delta_r:\G\to\G$, defined as 
$\delta_r=r^iv$ for all $r>0$ and $v\in H_i$. They constitute a one-parameter group of Lie group homomorphisms.  
We say that $d:\G\times\G\to[0,+\infty)$ is a \emph{homogeneous distance} on $\G$, if it is a continuous distance on $\G$ satisfying
the additional conditions
$$
d(z x, z y) = d(x, y)   \quad \mbox{and } \quad d(\delta_r(x), \delta_r(y)) = r d(x, y) \quad \mbox{for all } x,y,z \in \G, \, r > 0.$$
Let us denote by 
\begin{equation}\label{balls}
    B(x,r) = \{y \in \G \ : \ d(x,y) < r\} \quad \mbox{and } \; \B(x,r) = \{y \in \G \ : \ d(x,y) \leq r\}
\end{equation}
the open and closed balls with respect to a homogeneous distance $d$, respectively. 
We say that $\G$ is a  \emph{homogeneous group} if its Lie algebra is graded, and it is equipped with the above family of dilations $(\delta_\lambda)_{\lambda > 0}$ and a homogeneous distance $d$.
The integer $n$ denotes the dimension of $\G=H_1\oplus\cdots\oplus H_\iota$ as vector space.
Throughout the paper, $\G$ denotes a homogeneous group with direct decompositions $H_1\oplus\cdots H_\iota$, if not otherwise stated.

\begin{definition}
A graded basis $(e_1,  \ldots, e_n)$ of a homogeneous group $\G$ is a basis of vectors such that 
$$
(e_{m_{j-1} + 1}, \ldots, e_{m_j}) \quad \mbox{is a basis of } H_j \; \mbox{for each } j=1,\ldots,\iota,
$$
where we have defined the integers
$$
m_0 = 0, \quad m_j = \sum_{i=1}^j n_i \quad \mbox{and }\ n_j \coloneqq {\rm dim}(H_j) \quad \mbox{for every } j=1, \ldots, \iota,
$$
and $n_1 + \ldots + n_\iota$ coincides with the linear dimension $n$ of $\G$.
A graded basis provides
the {\em graded coordinates} $x = (x_1, \ldots , x_n) \in \mathbb{R}^n$, defining the unique element $x = \sum_{i=1}^n x_i e_i \in \G$. The unique left invariant vector fields 
$X_i\in\Lie(\G)$ such that $X_i(0)=e_i$ automatically define a 
{\em graded basis} $(X_1,\ldots,X_n)$ of $\Lie(\G)$.
\end{definition}

We additionally consider a left invariant Riemannian metric $g$ that makes orthonormal the fixed graded basis
$(X_1,\ldots,X_n)$ of $\Lie(\G)$. We say that 
this Riemannian metric $g$ is {\em graded}.
Such a metric also extends to $\Lambda_k({\rm Lie}(\G))$ and we denote by $| \cdot |_g$ the associated norm. The same metric defines a scalar product on $T_0\G$. Since $\G$ also has a linear structure, such scalar product induces a scalar product $|\cdot|$ on $\G$ via the identification of $\G$ with $T_0\G$. 
In the sequel, we assume that a graded Riemannian metric and the scalar product induced on $\G$ are fixed.

\subsection{Degree of submanifolds and the homogeneous tangent space} \label{sect3.2degree}
We fix a graded basis $(X_1,\ldots,X_n)$ of $\Lie(\G)$, along with the
associated graded basis $(e_1,e_2,\ldots,e_n)$ of $\G$,
and graded coordinates $(x_1,\ldots,x_n)$.

Next we introduce the degree of $k$-vectors, following \cite{Mag13Vit}.

\begin{definition}[Degree of $k$-vectors and of $k$-vector fields]
For each $1 \leq i \leq n$
we consider the unique integer $d_i\in\{1,\ldots,\iota\}$ such that $m_{d_i - 1} < i \leq m_{d_i}$. We say that
$d_i$ is the {\em degree} of $x_i$, of $e_i$, and of the left invariant vector field $X_i$, where we write $d(X_i)=d_i$. We denote by $I_{k,n}$ the family of all multi-index $I = (i_1, \ldots, i_k) \in \{1,\ldots, n\}^k$ such that $1 \le i_1 <i_2< \ldots < i_k \le n$. 
We define the \emph{$k$-vector} 
\begin{equation}\label{generaldefpvect}
    X_I = X_{i_1} \wedge \ldots \wedge X_{i_{k}} \in \Lambda_k({\rm Lie}(\G)),
\end{equation}
for each $I \in I_{k,n}$ we define the degree ${\rm deg}(X_I) = d_{i_1} + \ldots + d_{i_k}$. 
\end{definition}

If we fix a point $p\in\G$, the space $\Lambda_k({\rm Lie}(\G))$ can be also identified 
with both the space of \emph{left invariant $k$-vector fields} and with $\Lambda_k(T_p\G)$. The isomorphism can be obtained by considering first 
the isomorphism $J_p:\Lie(\G)\to T_p\G$,
$J_p(Y)=Y(p)$, hence we have the new isomorphism
$$
(\Lambda_kJ_p):\Lambda_k({\rm Lie}(\G))\to \Lambda_k(T_p\G)
$$
such that 
$(\Lambda_kJ_p)(Y_1\wedge Y_2\wedge \cdots \wedge Y_k)=Y_1(p)\wedge Y_2(p)\wedge \cdots \wedge Y_k(p)\in\Lambda_k(T_p\G)$
for every simple $k$-vector 
$Y_1\wedge Y_2\wedge \cdots \wedge Y_k\in\Lambda_k(\Lie(\G))$.

\begin{definition}\label{tangentpvect}
Let $\G=H_1\oplus\cdots \oplus H_\iota$ be a homogeneous group and let $p\in\G$. 
For $Y_1,\ldots,Y_k\in\Lie(\G)$, we consider a simple $k$-vector
$$Y_1 \wedge \ldots \wedge Y_k \in \Lambda_k(\Lie(\G),
$$
and the fixed graded basis $(X_1,\ldots,X_n)$ of $\Lie(\G)$. 
Then there exists unique coefficients 
$c_I\in\R$ such that 
$$
Y_1 \wedge \ldots \wedge Y_k = \sum_{I \in I_{k,n}} c_I X_I \in \Lambda_k({\rm Lie}(\G))
$$
and we define the {\em degree} of $Y_1\wedge Y_2\wedge\cdots \wedge Y_k$ as 
\[
\deg(Y_1\wedge\cdots\wedge Y_k)=\max\{\deg(X_I): c_I\neq0 \}.
\]
It can be checked that this definition does not depend on the graded basis $(X_1,\ldots,X_n)$.
Now we consider $v_1,v_2,\ldots v_k \in T_p\G$, the simple $k$-vector 
$v_1\wedge v_2\wedge \cdots \wedge v_k\in\Lambda_k(T_p\G)$ and we
can define 
\begin{equation}\label{degkvect}
\deg(v_1\wedge v_2\wedge \cdots\wedge v_k)=
\deg\Big((\Lambda_kJ_p)^{-1}(v_1\wedge v_2\wedge \cdots\wedge v_k)\Big).
\end{equation}
\end{definition}
It can be checked that the notion of degree 
\eqref{degkvect} does not depend on the choice of the graded
basis $(X_1,\ldots,X_n)$ of $\Lie(\G)$.
Following the notation in \eqref{degkvect}, 
we notice that choosing $Z_i\in\Lie(\G)$
such that $Z_i(p)=v_i\in T_p\G$ for all $i=1,\ldots,k$, namely $J_p^{-1}(v_i)=Z_i$, 
we have 
\[
(\Lambda_kJ_p)^{-1}(v_1\wedge v_2\wedge \cdots\wedge v_k)=
J_p^{-1}(v_1)\wedge J_p^{-1}(v_2)\wedge \cdots\wedge J_p^{-1}(v_k)=Z_1\wedge \cdots \wedge Z_k,
\]
hence we also have $\deg(v_1\wedge v_2\wedge \cdots\wedge v_k)=
\deg(Z_1\wedge Z_2\wedge \cdots \wedge Z_k)$.
We finally remark that the previous notions of degree for simple left invariant $k$-vector fields and for simple $k$-vectors of $T_p\G$ can be automatically extended to $k$-vectors.

We are now in the position to introduce the notion of degree and of pointwise degree for submanifolds, see for instance \cite{Magnani2019Area, Mag13Vit} for more information.

\begin{definition}\label{tangentpvectSigma}
 Let $\Sigma\subset\G$ be a $C^1$ smooth $k$-dimensional submanifold of a homogeneous group $\G$.
    A \emph{tangent k-vector} to $\Sigma$ at $p\in\Sigma$ is
$$ \tau_\Sigma(p) = t_1 \wedge \ldots \wedge t_k \in \Lambda_k(T_p \Sigma),$$
where $(t_1, \ldots, t_k)$ is a basis of $T_p\Sigma$.
We define the \emph{pointwise degree} of $\Sigma$ at $p\in\Sigma$ as 
\begin{equation}\label{degree}
d_\Sigma(p) =\deg\big(\tau_\Sigma(p)\big)
\end{equation} 
and the \emph{degree} of $\Sigma$ as the 
positive integer $d(\Sigma) = \max_{p \in \Sigma} d_\Sigma(p)$.
A point $p \in \Sigma$ has {\em maximum degree} when $d_\Sigma(p) = d(\Sigma)$.
\end{definition}

The notion of degree allows us to easily introduce {\em horizontal submanifolds} as those 
smooth submanifolds whose degree equals their topological dimension. In fact, horizontal submanifolds are characterized by having their tangent bundle contained in the horizontal subbundle of the group, see for instance \cite{Mag14} for a characterization of horizontal submanifolds.

\begin{definition}[Homogeneous tangent space]\label{homogene} Let $p \in \Sigma$ and set $d_\Sigma(p) = N$. We consider a tangent $k$-vector $\tau_\Sigma(p)$ to $\Sigma$ at $p$, along with the unique left invariant $k$-vector field $\xi_{p,\Sigma}$ such that $\xi_{p,\Sigma}(p) = \tau_\Sigma(p)$. Then we have the
unique coefficients $c_{p,I}\in\R$ such that
$$
 \xi_{p,\Sigma}= \sum_{I\in I_{k,n}} c_{p,I} X_I\in\Lambda_k(\Lie(\G))
$$
and we accordingly define the component
\[
\xi_{p,\Sigma,N}=\!\!\!\!\!\! \sum_{\substack{I\in I_{k,n}, \\ \deg(X_I)=N}} c_{p,I} X_I
\]
of $\xi_{p,\Sigma}$ having degree $N$. Since $d_\Sigma(p)=N$, we have $\xi_{p,\Sigma,N}\neq0$. Using \cite[Corollary 3.6]{Mag13Vit}, 
one easily notices that the following definition of \emph{Lie homogeneous tangent space} of $\Sigma$ at $p$
$$
\mathcal{A}_p \Sigma = \{ X \in {\rm Lie}(\G) \, : \, X \wedge \xi_{p,\Sigma,N} = 0 \}
$$
is well posed.
\end{definition}

\subsection{Intrinsic measure and measure-theoretic area formula}
We follow Section 2.10.1 of \cite{Federer69} to recall the general construction of a measure arising 
from a gauge function $\zeta : \mathcal{F} \rightarrow [0,+\infty]$. We have denoted by $\mathcal{F} \subset \mathcal{P}(\G)$ any nonempty family of closed subsets of a homogeneous group $\G$. For $\delta > 0$ and $A \subset \G$, we define
\begin{equation*}
    \phi_{\delta, \zeta}(A) = \inf \Biggl\{ \sum_{j=0}^\infty \zeta(B_j) \, : \, A \subset \bigcup_{j=0}^\infty B_j, \;\ {\rm diam}(B_j) \leq \delta, \; B_j \in \mathcal{F} \Biggr\}.
\end{equation*}
Considering $\phi_\zeta(A) = \sup_{\delta > 0} \phi_{\delta, \zeta}(A)$, we get a Borel regular measure $\phi_\zeta$ over the metric space $\G$. We introduce a specific gauge
$$
\zeta_\alpha(S) = ({\rm diam}(S)/2)^\alpha \quad \mbox{for every $S \subset \G$.}
$$
If we choose $\mathcal{F}$ to be the family $\mathcal{F}_b$ of all closed balls of $\G$ with positive radius and $\widetilde\zeta_\alpha = \restr{\zeta_\alpha}{\mathcal{F}_b}$, then $\phi_{\widetilde\zeta_\alpha}$ is the $\alpha$-\emph{dimensional spherical measure},
usually denoted by $\mathcal{S}^\alpha$. 
 In the case $\mathcal{F}$ coincides with the family of all closed sets and $k \in \{1, \ldots, n-1\}$, we define the Hausdorff measure
 \begin{equation}\label{hausmeas}
 \mathcal{H}^k_{| \cdot |}(E) = \mathcal{L}^k(\{x \in \R^k\, : \, \lvert x \rvert_E \le 1\}) \sup_{\delta > 0} \phi_\delta^k(E),
 \end{equation}
where $|\cdot|$ is the norm arising from a scalar product on $\G$, $\mathcal{L}^k$ is the standard Lebesgue measure on $\R^k$ and $|\cdot |_E$ is the Euclidean norm. Now, following \cite{Magnani2019Area}, we give the fundamental notion of intrinsic measure. 

\begin{definition}
Let $\Sigma \subset \G$ be a $k$-dimensional $C^1$ submanifold of degree $\NN$ and let $g$ be our fixed graded Riemannian on $\G$. We also choose an arbitrary Riemannian metric
$\tilde g$ on $\G$. Let $\tau_\Sigma$ be a tangent $k$-vector field on $\Sigma$ 
such that 
$$        \lvert \tau_\Sigma(p) \rvert_{\tilde g} = 1 \quad \mbox{for each } p \in \Sigma.
$$
Following the notation of Definition~\ref{homogene}, we define
\begin{equation}\label{Ntgvf}
    \tau_{\Sigma, N}^{\tilde g}(p) \coloneqq \xi_{p,\Sigma,N}(p) \quad \mbox{for each } \; p \in \Sigma.
\end{equation}
Then we define the \emph{intrinsic measure} of $\Sigma$ in $\G$ as
\begin{equation}\label{intrinsicmeas}
    \mu_\Sigma = \lvert\tau_{\Sigma, N}^{\tilde g} \rvert_g \, \sigma_{\tilde g},
\end{equation}
where $\sigma_{\tilde g}$ is the $k$-dimensional Riemannian measure induced by $\tilde g$ on $\Sigma$. 
\end{definition}
We observe that the tangent $k$-vector field $\tau_\Sigma$ need not be continuous in the case $\Sigma$ is not orientable.

\begin{definition} We fix $a > 0$, $p \in \G$ and consider a Borel regular measure $\mu$ over $\G$. We define the \emph{spherical Federer density} $\theta^a(\mu,\cdot)$ as 
\begin{equation}\label{Federerdens}
   \theta^a(\mu,  p) = \inf_{\varepsilon > 0} \sup \Bigl\{ \frac{2^a \mu(\B)}{{\rm diam}(\B)^a} \, : \, p \in \B \in \mathcal{F}_b, \; {\rm diam}(\B) < \varepsilon \Bigr\}.
   \end{equation}
\end{definition}
The previous definition was first introduced in \cite{Mag30} to establish various measure-theoretic area formulas, including the one with the spherical measure \cite[Theorem~11]{Mag30}. In the next theorem, we use a slightly more general version, which adapts \cite[Theorem 5.7]{LecMag22} to the case of homogeneous groups. 

\begin{theorem}\label{teo:areageneral}
Let $\mu$ be a regular measure and a Borel measure over a homogeneous group $\G$, 
and let $a>0$. We fix a Borel set $A \subset \G$ and
assume that the following conditions hold:
\begin{enumerate}
    \item $A$ has a countable covering whose elements are open and have $\mu$-finite measure,
    \item the subset $\{x \in A : \theta^a(\mu,x) = 0\}$ is $\sigma$-finite with respect to $\mathcal{S}^a$,
    \item $\mu \mres A$ is absolutely continuous with respect to $\mathcal{S}^a \mres A$.
\end{enumerate}
Then $\theta^a(\mu, \cdot) : A \to [0,+\infty]$ is Borel and for every Borel set $B \subset A$ we have
\begin{equation*}
        \mu(B) = \int_B \theta^a(\mu, p) d \mathcal{S}^a(p).
    \end{equation*}
\end{theorem}

\begin{definition}[Spherical factor]
    Let $V \subset \G$ be a $k$-dimensional linear subspace
    and let us consider a homogeneous distance $d$ on a homogeneous group $\G$. Denoting by $|\cdot|$ 
    our fixed graded scalar product on $\G$, 
    the \emph{spherical factor of $d$ with respect to $V$} is the number 
\begin{equation}\label{sfe}
    \beta_d(V) = \max_{d(x,0) \leq 1} \mathcal{H}_{| \cdot |}^k \big( \B(x,1) \cap V\big),
\end{equation}
where $\mathcal{H}_{|\cdot|}^k$ is given as in \eqref{hausmeas}.
\end{definition}

\section{Some general results for the upper blow-up in homogeneous groups}\label{sect:genres}

The present section contains two important tools to establish the upper blow-up of a submanifold in a homogeneous group. The first one is Theorem~\ref{t:specialcoord}, 
that we state using a specific system of graded coordinates $F:\R^{n}\to\G$, as in \cite[Corollary 3.8]{Mag13Vit}.
The second important tool is a suitable local expansion of the submanifold, that represents a sufficient condition to have the upper blow-up, see Theorem~\ref{genres}.

\begin{theorem}\label{t:specialcoord} Let $\G$ be a homogeneous group equipped
	with a graded scalar product and consider a $C^1$ smooth submanifold
	$\Sigma\subset\G$ containing the origin $0\in\G$ and of topological
	dimension $k$. Then there exist the integers $\alpha_{1},\ldots,\alpha_{\iota}\geq0$
	with $\alpha_{j}\le n_{j}$ for all $j=1,\ldots,\iota$, an orthonormal
	graded basis $(Y_{1},\ldots,Y_{n})$ of $\Lie(\G)$, a bounded open
	neighborhood $U\subset\R^k$ of the origin and a $C^{1}$ smooth
	embedding $\Phi:U\to\Sigma$ with the following properties. Defining
	the analytic diffeomorphism $F:\R^{n}\to\G$, $F(y)=\exp(\sum_{j=1}^{n}y_{j}Y_{j})$,
	there holds $\Phi(0)=0\in\G$, for all $u\in U$ 
	\[
	F^{-1}\circ\Phi(u)=\sum_{j=1}^{n}\phi_{j}(u)e_{j},\quad\phi(u)=(\phi_{1}(u),\ldots,\phi_{n}(u))
	\]
	and the Jacobian matrix of $\phi$ at the origin is 
	\begin{equation}
		D\phi(0)=\left(\begin{array}{c|c|c|c|c|c}
			I_{\alpha_{1}} & 0 & \cdots & \cdots & \cdots & 0\\
			0 & \ast & \cdots & \cdots & \cdots & \ast\\
			\hline 0 & I_{\alpha_{2}} & 0 & \cdots & \cdots & 0\\
			0 & 0 & \ast & \cdots & \cdots & \ast\\
			\hline 0 & 0 & I_{\alpha_{3}} & 0 & \cdots & 0\\
			0 & 0 & 0 & \ast & \cdots & \ast\\
			\hline \vdots & \vdots & \vdots & \ddots & \ddots & \vdots\\
			\hline 0 & 0 & \cdots & \cdots & \cdots & I_{\alpha_{\iota}}\\
			0 & 0 & \cdots & \cdots & \cdots & 0
		\end{array}\right).\label{e:matriceC}
	\end{equation}
	The blocks containing the identity matrix $I_{\alpha_{j}}$ have $n_{j}$
	rows, for every $j=1,\ldots,\iota$. The blocks $\ast$ are $(n_{j}-\alpha_{j})\times\alpha_{i}$
	matrices, for all $j=1,\ldots,\iota-1$ and $i=j+1,\ldots,\iota$.
	The mapping $\phi$ can be assumed to have the special graph form
	given by the conditions 
	\[
	\phi_{s}(u)=u_{s-m_{j-1}+\mu_{j-1}}
	\]
	for every $s=\m_{j-1}+1,\ldots,\m_{j-1}+\alpha_{j}$ and $j=1,\ldots,\iota$,
	where we have defined 	\begin{equation}\label{def:mu_j}
	\mu_{0}=0\qquad\mu_{j}=\sum_{i=1}^{j}\alpha_{i}\quad\text{for \ensuremath{j=1,\ldots,\iota.}}
	\end{equation}
	Moreover, representing the vector fields $Y_{j}$ in $\R^{n}$ by
	the coordinates given by $F$, we get
	\begin{equation}\label{basepotenziata}
	\widetilde{Y}_{j}=(F^{-1})_{*}(Y_{j})=\der_{y_{j}}+\sum_{\ell=m_{d_{j}}+1}a_{j \ell}(x)\der_{y_{\ell}},
	\end{equation}
	where $a_{j \ell}$ are homogeneous polynomial of degree $d_{\ell}-d_{j}$ with respect to
	the dilations of the group.
\end{theorem}
\begin{proof}
	We consider any $C^{1}$ smooth parametrization $\Psi:U'\to\Sigma$,
	that is a diffeomorphism onto $\Psi(U')=\Sigma\cap V$, where $V\subset\G$
	is an open set containing $0$. Due to \cite[Lemma 3.1]{Mag13Vit}
	there exists a matrix $C=(C_{ij})_{i=1,\ldots,n,j=1,\ldots,k}$
	having the form \eqref{e:matriceC} and a graded basis $(Y_{1},\ldots,Y_{n})$
	of $\Lie(\G)$ such that 
	\[
	\sum_{i=1}^{n}C_{ij}Y_{i}(e)=\sum_{s=1}^{k}\sigma_{sj}\der_{y_{s}}\Psi(0)
	\]
	for the unique coefficients $\sigma_{ij}\in\R$, where $j=1,\ldots,k$
	and $e\in\G$ is the unit element. We write $\Psi=F\circ\psi$, where
	$\psi:U'\to\R^{n}$, therefore observing that
	\[
	\sum_{i=1}^{n}C_{ij}Y_{i}(e)=\sum_{s=1}^{k}\sum_{i=1}^{n}\sigma_{sj}dF(0)(e_{i})(\der_{y_{s}}\psi^{i}(0))=\sum_{i=1}^{n}\left(\sum_{s=1}^{k}(\der_{y_{s}}\psi^{i}(0))\sigma_{sj}\right)Y_{i}(e),
	\]
	where $(e_{i})$ is the canonical basis of $\R^{n}$. Introducing
	the $k\times k$ invertible matrix $\mathcal{C}=(\sigma_{sj})$, we have proved
	the condition 
	\[
	D\psi(0)\,\mathcal{C}=C.
	\]
	Now we set $\widetilde{\psi}(y)=\psi(\mathcal{C}(y))$, observing that clearly
	\[
	D\widetilde{\psi}(0)=D\psi(0)\mathcal{C}=C.
	\]
	Let us define the projection 
	\[
\pi:\R^{n}\to\R^{k},\quad\pi\left(\sum_{j=1}^{n}y_{j}e_{j}\right)=\sum_{i=1}^{\iota}\sum_{s=m_{i-1}+1}^{m_{i-1}+\alpha_{i}}y_{s}E_{s-m_{i-1}+\mu_{i-1}},
	\]
	where $(E_{i})$ is the canonical basis of $\R^k$. Then $D(\pi\circ\widetilde{\psi})(0)$
	is the identity matrix and the inverse mapping theorem proves that $\pi\circ\widetilde{\psi}$ is a local diffeomorphism on a neighborhood
	of $0$. Setting the new variables $z=\pi\circ\widetilde{\psi}$,
	by construction  $\phi=\widetilde{\psi}\circ(\pi\circ\widetilde{\psi})^{-1}$
	has a graph form, and then the mapping $\Phi(y)=F\circ\phi(y)$
	satisfies all the assumptions of our claim. Finally, due to formula
	(2.42) in \cite{MagPhD} the proof is complete.
\end{proof}

We notice that \cite[Corollary~3.8]{Mag13Vit} is included in the previous theorem and both of these results follow from \cite[Lemma 3.1]{Mag13Vit}. The lemma provides a special adapted frame of vector fields on a submanifold, without referring to some system of coordinates. 
Theorem~\ref{t:specialcoord} can be also seen as a version of \cite[Theorem~3.1]{Magnani2019Area}, where the homogeneous group is identified with a vector space equipped with
the Lie product and the group operation given by the BCH formula.

The next theorem is an important tool for the upper blow-up, providing a general sufficient condition.
Its proof is the same as the proof of \cite[Theorem~1.2]{Magnani2019Area}, along with the observation that it only relies on the local expansion \eqref{localexp}. 
In Theorem~1.2 of \cite{Magnani2019Area} the local expansion was proved in all cases considered in the assumptions.

\begin{theorem}[Upper blow-up]\label{genres}
Let $\Sigma \subset \mathbb{G}$ be a smooth $k$-dimensional submanifold of class $C^1$ and 
degree $\NN$. Let $p \in \Sigma$ be a point of maximum degree $\NN$. 
 For the translated submanifold
$
\Sigma_p = p^{-1}\Sigma,
$
we introduce the $C^1$ smooth homeomorphism $\eta : \R^k \rightarrow \R^k$ by
\begin{equation}\label{funzioneeta}
\eta(t) = \bigg(\frac{\lvert t_1 \rvert^{b_1}}{b_1}{\rm sgn}(t_1), \ldots, \frac{\lvert t_k \rvert^{b_k}}{b_k}{\rm sgn}(t_k)\bigg)
\end{equation}
where $b_i$ is the induced degree, defined as follows
\begin{equation}\label{inducedb}
    b_i = j \quad \mbox{if and only if } \mu_{j - 1} < i \leq \mu_j
\end{equation}
for every $i=1 \ldots,k$,
where $\mu_j$ are defined in \eqref{def:mu_j}.
If $\phi$ denotes the mapping of Theorem~\ref{t:specialcoord}
applied to 
$\Sigma_p$, we define the $C^1$ smooth mapping
\begin{equation}\label{funzionegamma}
  \Gamma = \phi \circ \eta  
\end{equation}
and we set the subset of indexes $I \subset \{1,\ldots, n\}$ such that
\begin{equation*}
        A_0 \Sigma_p = { \rm span} \{ 
        e_1, \ldots, e_{\alpha_1}, e_{m_1 + 1}, \ldots, e_{m_1 + \alpha_2}, \ldots, e_{m_{\iota- 1} + 1}, \ldots, e_{m_{\iota -1} + \alpha_\iota}
        \}.
\end{equation*} 
Using notation and definitions of \cite[Theorem~3.1]{Magnani2019Area}, 
	if we assume the validity of the local expansion   
	\begin{equation}\label{localexp}
		\Gamma_s(t) = 
		\begin{cases}
			\frac{\lvert t_{s - m_{d_s - 1} + \mu_{d_s - 1}}\rvert^{d_s}}{d_s} {\rm sgn}(t_{s - m_{d_s - 1} + \mu_{d_s - 1}}) \;\; \mbox{if } s \in I \\
			o(\lvert t \rvert^{d_s}) \quad \quad \mbox{if } s \not \in I
		\end{cases},
	\end{equation}
	and $\beta_d \big(A_p \Sigma \big)$ denotes the \emph{spherical factor} associated to the homogeneous tangent space $A_p \Sigma$ in \eqref{sfe}, then
	\begin{equation}\label{eq:betadens}
	\theta^N(\mu_\Sigma, p) = \beta_d \big(A_p \Sigma \big).
	\end{equation}
\end{theorem}

\section{The Engel group with respect to a general graded basis}\label{generallambda}

In this section, we introduce the Engel group, and a special system of coordinates to represent its left invariant vector fields, the dual basis of left invariant 1-forms and the structure of tangent 2-vector fields to surfaces with a special parametrization, see \eqref{2vector}.

Up to Lie group isomorphisms, the Engel group is the unique connected and simply connected Lie group $\E$ whose Lie algebra $\cE$ admits a basis $(X_1,X_2,X_3,X_4)$, where the only nontrivial bracket relations
are 
\[
[X_1,X_2]=X_3\quad \text{and}\quad [X_1,X_3]=X_4.
\]
We have a grading $\cE=\Lie(\E)=V_1\oplus V_2\oplus V_3$, such that $V_1=\spn\{X_1,X_2\}$, $V_2=\spn\{X_3\}$, and
$V_3=\spn\{X_4\}$. 

\begin{proposition}[Change of coefficients in the parametrization]\label{formulasvector}
We consider a local parametrization $\Phi : U \rightarrow \G$ of a $k$-dimensional submanifold $\Sigma$, where $U\subset\R^k$ is open. We fix a graded basis 
$(Y_1,\ldots,Y_n)$ of ${\rm Lie}(\G)$ and consider the associated system of graded coordinates
$F(y)=\exp (\sum_{j=1}^n y_jY_j)$, hence we set $\phi=F^{-1}\circ \Phi:U\to\R^n$. 
We observe that there exist coefficients $a_{jl}$, that are homogeneous polynomials (with respect to dilations) such that 
\[
\widetilde{Y}_j(y)=\sum_{l=1}^n a_{jl}(y)\,\partial_{y_l}, \quad \mbox{for } j=1,\ldots,n,
\]
where $\widetilde{Y}_j = (F^{-1})_* (Y_j)$,  see for instance \cite[(2.42)]{MagPhD}.
We define the matrix $A=(a_{ij})$ and $c_i^l$ such that 
$$
\partial_{y_i} \Phi = \sum_{l=1}^n c_i^l(\phi)\,Y_l(\Phi) \quad \mbox{for } i=1,\ldots,k,
$$
or equivalently $\partial_{y_i} \phi = \sum_{l=1}^n c_i^l(\phi)\,\widetilde Y_l(\phi)$ for 
$i=1,\ldots,k$. Then we have the formula
\begin{equation}\label{coefficientspartial}
c_j^i(\phi)=\sum_{l=1}^n((A(\phi)^T)^{-1})_{il}\,\der_{y_j}\phi^l.
\end{equation}
\end{proposition}
For the sake of completeness, we add the proof of this proposition, although it is a straightforward verification.
\begin{proof} 
	We simply write down our definitions
	\begin{align*}
		\partial_{y_j} \phi &= \partial_{y_j}(F^{-1} \circ \Phi) =d F^{-1}(\Phi) (\partial_{y_j} \Phi) \\
		& =d F^{-1}(\Phi) \Big(\sum_{l=1}^n c_j^l(\phi) Y_l(\Phi)\Big) = \sum_{l=1}^n c_j^l(\phi) d F^{-1}(\Phi) ( Y_l(\Phi)) \\
		&= \sum_{l=1}^n c_j^l(\phi) \widetilde{Y}_l(F^{-1} \circ \Phi) 
		=\sum_{l=1}^n c_j^l(\phi) \widetilde{Y}_l(\phi)
		=\sum_{l=1}^n c_j^l(\phi) \sum_{k=1}^n\, a_{lk}(\phi)\partial_{y_k} \\
		&= \sum_{k=1}^n\sum_{l=1}^n(A^T(\phi))_{kl}   c_j^l(\phi)  \partial_{y_k}
		= \sum_{k=1}^n (\partial_{y_j} \phi^k) \partial_{y_k}.
	\end{align*}
	Thus we can infer that
	$$
	\partial_{y_j} \phi^k = \sum_{l=1}^n(A^T(\phi))_{kl}   c_j^l(\phi)
	$$
	hence taking the inverse matrix of $A^T(\phi)$, our claim \eqref{coefficientspartial} follows.   
\end{proof}

We start from any graded basis $(Y_1,Y_2,Y_3,Y_4)$ of $\cE$, hence the grading
of $\cE$ implies the existence of real numbers $\xi_{12}, \xi_{13},\xi_{23}$ 
such that 
\begin{equation}\label{structurecoeff}
[Y_1, Y_2] = \xi_{12} Y_3,  \quad [Y_1, Y_3] =\xi_{13} Y_4, \quad [Y_2, Y_3] = \xi_{23}Y_4.
\end{equation} 
The stratification of $\cE$ gives $\xi_{12} \neq 0$ and $(\xi_{13}, \xi_{23}) \neq (0,0)$.

By the Baker-Campbell-Hausdorff formula, we can construct the general (or parametrized) frame of vector fields $\widetilde Y_1,\widetilde Y_2,\widetilde Y_3,\widetilde Y_4$ of $\R^4$ that yield the Engel algebra and give a coordinate system such that the exponential mapping is the identity mapping and the inverse element is the opposite element, $y^{-1}=-y$. Thus, we introduce the coordinate system generated by the exponential map
\begin{equation}\label{analyticdiffeo}
F : \R^4 \rightarrow \G, \quad F(y) = \exp \Big(\sum_{j=1}^4 y_j Y_j\Big)
\end{equation}
and define the preimages 
\[
\widetilde Y_j=(F^{-1})_*(Y_j),
\]
that are vector fields of $\R^4$.
Arguing as in the proof of \cite[(2.42)]{MagPhD},
we take into account \eqref{structurecoeff} and differentiate the left translation associated with the BCH formula, obtaining an explicit formula for $(\widetilde Y_1,\widetilde Y_2,\widetilde Y_3, \widetilde Y_4)$, that verify the conditions of \eqref{structurecoeff} and then
generate a Lie algebra isomorphic to $\cE$. Thus, for this general form of a graded basis of $\cE$, we get
\begin{align}\label{basismain}
\notag \widetilde{Y}_1 &= \partial_{y_1} - \frac{y_2}{2} \xi_{12}\partial_{y_3} - \Big(\frac{y_3}{2} \xi_{13} +\frac{1}{12}\xi_{12}\xi_{13}y_1 y_2 + \frac{1}{12} \xi_{12}\xi_{23}y_2^2\Big) \partial_{y_4}, \\
\widetilde{Y}_2 &= \partial_{y_2} + \frac{y_1}{2}\xi_{12}\partial_{y_3}- \Big(\frac{y_3}{2}\xi_{23} -\frac{1}{12}\xi_{12}\xi_{23}y_2 y_1 - \frac{1}{12}\xi_{12}\xi_{13}y_1^2  \Big)\partial_{y_4}, \\ \notag
\widetilde{Y}_3 &= \partial_{y_3} +\frac12\big(y_1\xi_{13}+y_2\xi_{23}\big)\partial_{y_4}, \\ 
\notag \widetilde{Y}_4 &= \partial_{y_4}.
\end{align}
We also need to write the left invariant dual basis $(\theta_1, \ldots, \theta_4)$ of $(\widetilde{Y}_1, \ldots, \widetilde{Y}_4)$, obtaining
\begin{align}\label{dualbasisgeneral}
    &\theta_1=dy_1, \quad \theta_2=dy_2, \quad \theta_3 = d y_3 -  \xi_{12}\frac{y_1}{2} dy_2 + \xi_{12}\frac{y_2}{2} dy_1, 
    \end{align}
    and finally 
    \begin{align*}
     \theta_4&=dy_4 - \frac12\big(y_1\xi_{13}+ y_2\xi_{23}\big) dy_3 + \Big( \frac16 \xi_{12}y_1(\xi_{23}  y_2 + \xi_{13}y_1)  + \frac{y_3}{2}\xi_{23}\Big) dy_2 \\ 
     & + \Big(\frac{y_3}{2} \xi_{13} - \frac16\xi_{12}y_2(\xi_{13} y_1  + \xi_{23}y_2)\Big) dy_1.
    \end{align*}
Combining \eqref{basismain} and Proposition~\ref{formulasvector},
setting $u=(u_1, u_2)\in U$, we can explicitly compute 
\begin{align}\label{generalvect}
\notag \der_{u_i}\phi & 
\notag = \der_{u_i} \phi_1 \widetilde{Y}_1 + \der_{u_i} \phi_2 \widetilde{Y}_2 + \Big(\der_{u_i} \phi_3 - \frac{\xi_{12}}{2}\phi_1 \der_{u_i} \phi_2  + \frac{\xi_{12}}{2}\phi_2 \partial_{u_i} \phi_1\Big) \widetilde{Y}_3 + \\
    &
    \Bigg(\der_{u_i} \phi_4 - \frac12(\xi_{13}\phi_1+\xi_{23}\phi_2) \der_{u_i} \phi_3  + \Big(\frac{1}{6}\xi_{12}(\xi_{23}\phi_1 \phi_2 + \xi_{13} \phi_1^2)
    + \frac{\xi_{23}}{2}\phi_3\Big)\der_{u_i} \phi_2 \\
  \notag & 
+\Big(\frac{\xi_{13}}{2}\phi_3 -\frac{1}{6}\xi_{12}(\xi_{13}\phi_1 \phi_2 + \xi_{23}\phi_2^2)  \Big)\der_{u_i}\phi_1\Bigg) \widetilde{Y}_4.
\end{align}
where the vector fields $\widetilde{Y}_j$ are evaluated at $\phi(u)$.

Let us consider a $2$-vector associated to
the differential of $\phi$ with respect to $\widetilde{Y}_i$:
\begin{align}\label{2vector}
\notag	\der_{u_1}\phi \wedge \der_{u_2}\phi &= \phi_u^{12}\widetilde{Y}_1 \wedge \widetilde{Y}_2 + \big(\phi_u^{13} - \frac{\xi_{12}}{2}\phi_1  \phi_u^{12}\big) \widetilde{Y}_1 \wedge \widetilde{Y}_3+ (\phi_u^{23} - \frac{\xi_{12}}{2}\phi_2 \phi^{12}_u)\widetilde{Y}_2 \wedge \widetilde{Y}_3  \\ \notag
	&+\Big(\phi_u^{14}  - \frac12(\xi_{13}\phi_1+\xi_{23}\phi_2) \phi_u^{13} + \Big(\frac{1}{6}\xi_{12}(\xi_{23}\phi_1 \phi_2 + \xi_{13} \phi_1^2)
    + \frac{\xi_{23}}{2}\phi_3\Big)\phi_u^{12} \Big) \widetilde{Y}_1 \wedge \widetilde{Y}_4 \\
    &+
	\Bigg(\phi_u^{24} -\frac12(\xi_{13}\phi_1+\xi_{23}\phi_2) \phi_u^{23} + \big(\frac{1}{6}\xi_{12}(\xi_{13}\phi_1 \phi_2 + \xi_{23}\phi_2^2) - \frac{\xi_{13}}{2}\phi_3 \big)\phi^{12}_u \Bigg) \widetilde{Y}_2 \wedge \widetilde{Y}_4 \\ \notag
	&\notag+\Bigg(\phi_u^{34}+\Big(
 \frac{1}{12}\xi_{12}(\xi_{23}\phi_1 \phi_2 + \xi_{13} \phi_1^2)
    - \frac{\xi_{23}}{2}\phi_3\Big) \phi_u^{23} - \phi_1 \frac{\xi_{12}}{2}\phi^{24}_u + \phi_2 \frac{\xi_{12}}{2}\phi^{14}_u  \\ \notag
	&- \Big(\frac{1}{12}\xi_{12}(\xi_{13}\phi_1 \phi_2 + \xi_{23}\phi_2^2) + \frac{\xi_{13}}{2}\phi_3 
 \Big) \phi^{13}_u  +\Big( \frac{\xi_{12} \xi_{13}}{4} \phi_1 \phi_3 + \frac{\xi_{12} \xi_{23}}{4}\phi_2 \phi_3\Big) \phi^{12}_u \Bigg) \widetilde{Y}_3 \wedge \widetilde{Y}_4,
\end{align}
where we set
$$
\phi_u^{ij} = \det
\begin{pmatrix}
	\partial_{u_1}\phi_i & \partial_{u_2}\phi_i \\
	\partial_{u_1}\phi_j & \partial_{u_2}\phi_j
\end{pmatrix}.
$$

\section{Non-existence of horizontal surfaces in the Engel group}\label{sect:nonexistence}

In order to show the nonexistence of horizontal surfaces in the Engel group, we can fix any graded basis of its Lie algebra. We consider
\begin{align}\label{vectorfieldss}
\notag X_1 &= \partial_{x_1} - \frac{x_2}{2} \partial_{x_3} - \big( \frac{x_3}{2} + \frac{x_1 x_2}{12}\big) \partial_{x_4}, \\
X_2 &= \partial_{x_2} + \frac{x_1}{2} \partial_{x_3} + \frac{x_1^2}{12} \partial_{x_4},\\
\notag X_3 &= \partial_{x_3} + \frac{x_1}{2} \partial_{x_4},\\ \notag\quad X_4 &= \partial_{x_4},
\end{align}
that follows from \eqref{basismain} with $\xi_{12} = \xi_{13} = 1$ and $\xi_{23} = 0$.

The elements of the dual basis $\theta_1$, $\theta_2$, $\theta_3$ and $\theta_4$
are defined by
\begin{equation}\label{dualbasis}
    d x_1, \quad dx_2, \quad dx_3 -\frac{x_1}{2} dx_2 + \frac{x_2}{2} dx_1, \quad dx_4 -\frac{x_1}{2}dx_3 + \frac{x_1^2}{6}dx_2 + \Big(\frac{x_3}{2} - \frac{x_1 x_2}{6}\Big) dx_1,
\end{equation}
respectively, and we have
\[
    d \theta_3 = - dx_1 \wedge d x_2,  \quad d \theta_4 = - d x_1 \wedge d x_3 + \frac{1}{2} x_1 d x_1 \wedge d x_2.
\]
Following the approach of \cite{Mag2010pams}, in the next result we prove the nonexistence
of horizontal $C^1$ surfaces in the Engel group.

\begin{theorem}\label{horizontalnon}
There does not exist any $C^1$ smooth surface of degree $2$ in the Engel group.
\end{theorem}
\begin{proof}
By contradiction, we assume that we have a $C^1$ smooth surface $\Sigma \subset \E$ of degree $2$.
   Let $\Om \subset \mathbb{R}^2$ be an open set, let $\Phi=\exp(\sum_{j=1}^4 \phi_jX_j) \in C^{1}(\Om, \Sigma)$
   be a local parametrization of $\Sigma$ and define $\phi=(\phi_1,\phi_2,\phi_3,\phi_4)\in C^1(\Omega,\R^4)$. The assumption that $\Sigma$ has degree $2$ implies that $\phi^*(\theta_3) = \phi^*(\theta_4)=0$, 
that is
\begin{equation}\label{systems}
d (2\phi_3) = \phi_1 d \phi_2 - \phi_2 d \phi_1   \quad \mbox{and} \quad d \phi_4 = \big(\frac{\phi_1 \phi_2}{6} -\frac{\phi_3}{2}\big) d \phi_1 +\frac{\phi_1}{2} d \phi_3 -  \frac{\phi_1^2}{6} d \phi_2
\end{equation}
everywhere in $\Omega$. 
Since $\phi$ has everywhere maximal rank, taking into account the form of $d(2\phi_3)$ in \eqref{systems}, we may argue as in the proof
of \cite[Lemma 2]{Mag2010pams}, where we replace $f= (f_1, f_2, f_3)$ of this lemma by $(\phi_1, \phi_2, 2\phi_3)$ of the present theorem.
Then we conclude that the Jacobian $\det(((\phi_i)_{x_j})_{i,j=1,2})$ vanishes everywhere in $\Omega$. 
Combining this result with the first equation of \eqref{systems} we have that
the linear space spanned by $\nabla \phi_1$, $\nabla \phi_2$ and $\nabla \phi_3$ is one dimensional.
Finally, the second equation in \eqref{systems} allows us to conclude that the rank of $\nabla \phi$ is less than two everywhere in $\Om$, therefore
contradicting the assumption that $\Phi$ is a parametrization of $\Sigma$.
\end{proof}

\section{Examples of surfaces of degree three and four}\label{sect:examples_surfaces}

In this section, we are only concerned with 2-dimensional surfaces in the Engel group $\E$. From the area formula \eqref{areafo}, the degree of a sufficiently smooth surface in $\E$ coincides with its Hausdorff dimension. We will construct nontrivial surfaces of degree three and four in the Engel group. Notice that surfaces of degree five have the maximal Hausdorff dimension, hence they are very easy to construct. Actually they are transversal submanifolds, see \cite{Mag26TV,Magnani2019Area} for more information.

We fix a system of coordinates $(y_1,y_2,y_3,y_4)$ of the Engel group $\E$, such that the basis of the Lie algebra is \eqref{basismain} with $\xi_{12}=\xi_{13}=1$ and $\xi_{23}=0$, therefore
\begin{align}\label{vectorfieldsy}
\notag \widetilde Y_1 &= \partial_{y_1} - \frac{y_2}{2} \partial_{y_3} - \big( \frac{y_3}{2} + \frac{y_1 y_2}{12}\big) \partial_{y_4}, \\
\widetilde Y_2 &= \partial_{y_2} + \frac{y_1}{2} \partial_{y_3} + \frac{y_1^2}{12} \partial_{y_4},\\
\notag \widetilde Y_3 &= \partial_{y_3} + \frac{y_1}{2} \partial_{y_4},\\ \notag\quad \widetilde Y_4 &= \partial_{y_4}.
\end{align}
We consider a surface $\Sigma\subset\E$ parametrized by $\phi:U\to\Sigma$, where $U\subset\R^2$ is open and contains $(0,0)$. 
Then the tangent 2-vector field \eqref{2vector} to the surface $\Sigma\subset\E$ parametrized by $\phi$,  becomes 
\begin{eqnarray}\label{2vectorxi}
& \partial_{u_1}\phi \wedge \partial_{u_2}\phi = \phi_u^{12}\widetilde{Y}_1 \wedge \widetilde{Y}_2 + \big(\phi_u^{13} - \frac12\phi_1  \phi_u^{12}\big) \widetilde{Y}_1 \wedge \widetilde{Y}_3+ (\phi_u^{23} - \frac12\phi_2 \phi^{12}_u)\widetilde{Y}_2 \wedge \widetilde{Y}_3  \\
&+\Big(\phi_u^{14}  - \frac12\phi_1 \phi_u^{13} + \frac{1}{6}\phi_1^2
    \phi_u^{12} \Big) \widetilde{Y}_1 \wedge \widetilde{Y}_4 +
	\Bigg(\phi_u^{24} -\frac12\phi_1 \phi_u^{23} + \big(\frac{1}{6}\phi_1 \phi_2 - \frac12\phi_3 \big)\phi^{12}_u \Bigg) \widetilde{Y}_2 \wedge \widetilde{Y}_4 \nonumber \\ 
	&\notag+\Bigg(\phi_u^{34}+
 \frac{1}{12} \phi_1^2
  \phi_u^{23} - \phi_1 \frac12\phi^{24}_u + \phi_2 \frac12\phi^{14}_u - \Big(\frac{1}{12}\phi_1 \phi_2 + \frac12\phi_3 
 \Big) \phi^{13}_u  + \frac14 \phi_1 \phi_3 \phi^{12}_u \Bigg) \widetilde{Y}_3 \wedge \widetilde{Y}_4. 
\end{eqnarray}
Assumping $\phi(0,0)=0\in\R^4$ and considering the graph mapping $\phi(x,y)=(\phi_1,\phi_2,x,y)$, we have $\phi^{34}_u=1$.
Due to \eqref{2vectorxi}, the surface $\Sigma$, equal to the image of $\phi$, must have degree 5.

\subsection{Surfaces of degree 3}

We suppose that $\phi(x,y) = (\phi_1(x,y), x, y, \phi_4(x,y))$ and that the last three coefficients
of $\widetilde Y_i\wedge \widetilde Y_j$ in \eqref{2vectorxi} vanish, therefore 
\begin{equation}\label{eqfordegree3}
\begin{cases}
 \phi_u^{14} = \frac12 \phi_1 \big(\phi_u^{13} -\frac13 \phi_1 \phi_u^{12}\big) \\
 \phi_u^{24} = \frac12 \phi_1 \big(\phi_u^{23} -\frac13 \phi_2 \phi_u^{12}\big) +\frac12 \phi_3 \phi_u^{12} \\
 \phi_u^{34} = \frac{\phi_1}{2} \phi_u^{24} -\frac{\phi_2}{2} \phi_u^{14}-\frac{1}{12}\phi_1^2 \phi_u^{23} + \big(\frac{1}{12}\phi_1 \phi_2 + \frac12 \phi_3 \big)\phi_u^{13} -\frac14 \phi_1 \phi_3 \phi_u^{12}.
\end{cases}
\end{equation}
We clearly have 
$$
\phi_u^{12}= -\partial_y \phi_1 \quad  \phi_u^{13}=\partial_x \phi_1, \quad\phi_u^{23}=1, \quad  \phi_u^{24}=\partial_y \phi_4 \quad \mbox{and}\quad \phi_u^{34}= -\partial_x \phi_4.
$$
Inserting these equalities in \eqref{eqfordegree3}, we get that
\begin{equation}\label{systemm}
\begin{cases}
    \phi_u^{14} = \frac12 \phi_1 \partial_x \phi_1  +\frac16 \phi_1^2 \partial_y \phi_1 \\
    \partial_y \phi_4 = \frac12 \phi_1 +  \frac16 x \phi_1\partial_y \phi_1 - \frac12 y \partial_y \phi_1 \\
    \partial_x \phi_4 = -\frac12\phi_1 \partial_y \phi_4 +\frac{x}{2}\phi_u^{14}  + \frac{1}{12}\phi_1^2 - (\frac{1}{12}x \phi_1 + \frac12 y)\partial_x \phi_1 - \frac14 y \phi_1 \partial_y \phi_1.
\end{cases}
\end{equation}
By substituting the first two equations into the third one, we notice that 
$d\phi_4$ only depends on $\phi_1$ and its partial derivatives.
After some elementary computations, the existence condition $d^2\phi_4=0$ for $\phi_4$ follows by assuming that 
$\phi_1$ solves the Burgers' equation
\begin{equation}\label{burgers}
    \partial_x \phi_1 + \frac12 \phi_1 \partial_y \phi_1 = 0.
\end{equation}
We may choose a simple nontrivial solution
\begin{equation}
    \phi_1(x,y) = 2 \frac{1-y}{1-x},
\end{equation}
see for instance \cite{EvansPDE2010}, therefore $d\phi_4$ has an explicit formula. After some more computations $\phi_4$ can be computed, 
leading to the solution
\begin{equation}\label{esempiograd3}
\phi(x,y) = \Big(2\frac{1-y}{1-x}, x, y,\frac{3x^2+xy^2-5xy-2x+3y}{3(1-x)^2}\Big).
\end{equation}
The image of $\phi$ defines a surface of degree three. It is pretty clear that selecting different solutions of the
Burgers' equation \eqref{burgers} yields different solutions.

\subsection{Surfaces of degree 4} 
Again the tangent 2-vector field represented in \eqref{2vectorxi} shows that a surface $\Sigma$ of degree $4$ and parametrized by $\phi=(\phi_1,\phi_2,\phi_3,\phi_4)$ must satisfy
\begin{equation}\label{vincolgrado4}
\phi_u^{34} = \frac{\phi_1}{2} \phi_u^{24} -\frac{\phi_2}{2} \phi_u^{14}-\frac{1}{12}\phi_1^2 \phi_u^{23} + (\frac{1}{12}\phi_1 \phi_2 + \frac12 \phi_3)\phi_u^{13} -\frac14 \phi_1 \phi_3 \phi_u^{12}.
\end{equation}
Let us choose the special parametrization $\phi(x,y) = (x, \phi_2, xy, y)$, with $\phi_2(0,0)=0$, so that 
$$\phi_u^{34} = y, \; \phi_u^{23} = x \partial_x \phi_2 - y \partial_y \phi_2, \;  \phi_u^{24}=\partial_x \phi_2, \; \phi_u^{13} = x, \;  \mbox{and }\; \phi_u^{12} = \partial_y \phi_2$$
and \eqref{2vectorxi} implies that $d_\Sigma(0)\ge4$, since $\phi^{14}_u=1$.
Then \eqref{vincolgrado4} becomes
\[ 
y + \frac{\phi_2}{2} -\frac{1}{12}x^2 \phi_2 -\frac12x^2y = \big(\frac{x}{2} -\frac{1}{12}x^3\big)\partial_x\phi_2 -\frac16x^2y\partial_y\phi_2,
\]
and this implies that 
\begin{equation}\label{mdc0}
(6x - x^3)\partial_x\phi_2 -2x^2 y \partial_y\phi_2 = 12y + 6\phi_2 -x^2 \phi_2 - 6x^2y.
\end{equation}
Using the method of characteristics, we consider
\begin{equation}\label{mdc1}
    \dot x(t) = 6x(t) - x^3(t), \quad \dot y(t) = -2x^2(t) y(t)
\end{equation}
and impose the condition
\begin{equation}\label{mdc2}
    \dot\phi_2(t) = 12 y(t) + 6 \phi_2(t) - x^2(t) \phi_2(t) - 6 x^2(t) y(t).
\end{equation}
Taking into account \eqref{mdc1}, the simple ODE
\begin{equation*}
    \frac{dy}{dx}=-\frac{2x}{6-x^2}y
\end{equation*}
with $x \neq 0$ and $|x|\neq\sqrt{6}$ can be explicitly computed 
and a solution $y(x)$ can be inserted into \eqref{mdc2}.
Then $\Sigma=\phi(U)$ is an example of degree-4 surface given by
\begin{equation*}
\phi(x,y) = \Bigg(x, - \frac{6y(x^2 + 2)}{6 - x^2} + x\, F\left(\frac{y}{6 - x^2}\right), xy, y\Bigg)
\end{equation*}
for $x,y$ sufficiently close to the origin and any smooth function $F$.

\section{Upper Blow-up in the Engel group}

This section is devoted to the proof of the upper blow-up. 
First of all, we prove a rather natural property of the degree when we pass from 
the surface to its boundary.

\begin{lemma}\label{gradosulbordo}
If $\Sigma \subset \E$ is a $2$-dimensional $C^1$ smooth submanifold with boundary $\partial \Sigma$, then 
for every $p\in\partial \Sigma$ we have 
$d_{\partial \Sigma}(p) \leq d_{\Sigma}(p) - 1$ and in particular
    \begin{equation}\label{thesisgrado}
    d(\partial \Sigma) \leq d(\Sigma) - 1.
    \end{equation}
\end{lemma}
\begin{proof}
Let $p \in \partial \Sigma$ and choose two linearly independent tangent vectors $t_1 \in T_p \partial \Sigma\setminus\set{0}$ and $t_2 \in T_p \Sigma\setminus\set{0}$. Then we have a basis of $T_p \Sigma$ and we set 
$$
t_i = \sum_{j=1}^4 c_i^j\, Y_j(p) \quad \mbox{for } i=1,2,
$$ 
for some $c_i^j\in\R$ and a fixed graded basis $(Y_1,Y_2,Y_3, Y_4)$ of ${\rm Lie}(\E)$.
We consider the 2-vector
\[
t_1 \wedge t_2 = \sum_{1 \leq i < j \leq 4}^4 M^{ij}_{12}(C)\,  Y_i(p)\wedge Y_j(p)\in \Lambda_2(T_p\Sigma),
\]
where $M^{ij}_{12}(C)= c_1^i c_2^j - c_1^jc_2^i=-M_{12}^{ji}(C)$ and we have set 
\begin{equation}\label{casobordo}
C =
    \begin{pmatrix}
    c_1^1 & c_2^1 \\
    c_1^2 & c_2^2\\
    c_1^3 & c_2^3 \\
    c_1^4 & c_2^4 
\end{pmatrix}.
\end{equation}
Our claim corresponds to the inequality 
\begin{equation}\label{disuguaglianzagrado}
d_\Sigma(p)=\deg(t_1 \wedge t_2) >\deg(t_1)=d_{\partial \Sigma}(p). 
\end{equation}
Let us define $d_{i_0}=\deg(t_1)$ for some $i_0\in\{1,2,3,4\}$ and observe that $c_1^{i_0}\neq0$. We have two possible cases. If there exists
$j_0\in\{1,2,3,4\}\sm\{i_0\}$ such that 
$M_{12}^{i_0j_0}(C)\neq0$, then
\[
\deg(t_1\wedge t_2)\ge d_{i_0}+d_{j_0}>
d_{i_0}=\deg(t_1),
\]
proving our claim. 
We are left to consider the case where $M_{12}^{i_0j}(C)=0$ for all $j\in\set{1,2,3,4}\sm\set{i_0}$.
Such condition implies that the three rows of $C$ that differ from the $i_0$-th row, are all proportional to $(c_1^{i_0},c_2^{i_0})$, or some of them may vanish. In all these cases the rank of $C$ is less than two, hence contradicting the fact that $t_1\wedge t_2\neq0$.
This proves \eqref{disuguaglianzagrado}, hence also \eqref{thesisgrado} immediately follows.
\end{proof}

The next theorem is a crucial tool for our results. Its proof relies on Stokes' theorem.

\begin{theorem}\label{mainstokestheorem}
Let $\Sigma \subset \E$ be a $2$-dimensional $C^1$ smooth submanifold of degree 3. If $p\in\Sigma$ is a point of maximum degree $3$, 
then the orthonormal graded basis $(Y_1,Y_2,Y_3,Y_4)$ provided by Theorem~\ref{t:specialcoord} and having the
structure coefficients $\xi_{ij}$ given by \eqref{structurecoeff} must have $\xi_{13} = 0$.
\end{theorem}

\begin{proof}
By our assumptions $d(\Sigma) = d_\Sigma(p)= 3$.
Applying Theorem~\ref{t:specialcoord} and following its notation, we have $\alpha_1=\alpha_2 =1$ and $\alpha_3 = 0$. 
According to our claim, the same theorem provides us with the graded basis $(Y_1,Y_2,Y_3, Y_4)$
and a $C^1$ smooth parametrization $\Phi : U \rightarrow \E$ of $p^{-1}\Sigma$ around the origin, such that  
\begin{equation}\label{para4}
        \phi(x_1, x_3)=\big(x_1,\phi_2(x),x_3,\phi_4(x) \big)
\end{equation} 
is defined on an open neighborhood $U \subset \R^2$ of $0$,
we have the formula $\Phi=F\circ\phi$ and $F:\R^4\to\E$, $F(x)=\exp(\sum_{j=1}^4x_jY_j)$.
In particular, \eqref{e:matriceC} gives us
\begin{equation}\label{graphphi}
 D\phi=
\begin{pmatrix}
	1    & 0 \\
	o(1) & \star \\
	0  &  1\\
	o(1) & o(1)
\end{pmatrix}
\end{equation}
around the origin of $\R^2$. 
We recall that (see for instance \cite[Proposition~14.32]{Lee2013IntroSM}), if $V_1,V_2$ are general vector fields of $\E$, the exterior differential 
of a 1-form $\beta$ reads as 
\[
d\beta(V_1, V_2)=V_1\beta(V_2)-V_2\beta(V_1)-\beta([V_1,V_2]).
\]
Since $\theta_k$ and $Y_j$ are left invariant for $k,j=1,\ldots,4$, the previous formula gives
\[
d\theta_k(Y_i, Y_j)=-\theta_k([Y_i,Y_j])
\]
for all $i,j,k=1,2,3,4$, where $(\theta_1,\theta_2,\theta_3,\theta_4)$ is the dual basis of
$(Y_1,Y_2,Y_3,Y_4)$.
We notice that the only possibly nonvanishing values
of $d\theta_4$ on any couple of $(Y_i,Y_j)$ with $i<j$ are 
\begin{align}\label{relationdual}
d\theta_4(Y_i,Y_3)= -\theta_4([Y_i, Y_3]) =- \xi_{i3} (\theta_i \wedge \theta_3)(Y_i,Y_3)=-\xi_{i3},
\end{align}
for $i=1,2$. Observing that the previous 2-forms vanish on all other 2-vectors, we get 
\begin{equation}\label{eq:dtheta_4form}
d\theta_4 = -\xi_{13}\theta_1 \wedge \theta_3-\xi_{23}\theta_2 \wedge \theta_3
\end{equation}
that is the Maurer-Cartan equation for $d\theta_4$. 

Now, let $r > 0$ be sufficiently small such that $B_r \Subset U \subset \R^2$, where $B_r$ is the Euclidean ball of radius $r$ 
and centered at the origin. Let us consider the restricted surface
$\Sigma_r=\phi(B_r)$. Since $\phi$ is a graph map,
$\Sigma_r$ is a $C^1$ smooth, orientable surface with degree $d(\Sigma_r) = d(\Sigma)$ and orientable boundary $\partial \Sigma_r$. 

Applying the Stokes' theorem for $C^1$ smooth manifolds, we get the following equalities
\begin{equation}\label{fundstoke}
      0 = 
   \int_{\partial \Sigma_r} \theta_4 = \int_{\Sigma_r} d \theta_4.   
\end{equation}
The first equality follows by Lemma \ref{gradosulbordo}, observing that 
$d(\partial \Sigma_r)\leq 2$ and $\theta_4(Y_k)= 0$ for $k=1,2,3$. 
Using \eqref{fundstoke} and \eqref{eq:dtheta_4form}, we obtain
\begin{align*}
    0 &= \lim_{r \rightarrow 0}\frac{1}{\pi r^2} \int_{B_r} \phi^*(d\theta_4)=-\xi_{13}\phi^*(\theta_1 \wedge \theta_3)(0)-\xi_{23}
    \phi^*(\theta_2 \wedge \theta_3)(0) \\
   &=-\xi_{13}(d\phi_1 \wedge d\phi_3)(0)-\xi_{23}  \phi^*(d\phi_2 \wedge d\phi_3)(0)
\end{align*}
where the last equality follows from \eqref{dualbasisgeneral}. The form of the Jacobian of $\phi$ at the 
origin \eqref{graphphi} proves that $(d\phi_1 \wedge d\phi_3)(0)=1$ and $(d\phi_2 \wedge d\phi_3)(0)=0$,
therefore concluding the proof.
\end{proof}

The previous theorem allows us to prove the upper blow-up at points of maximum degree in surfaces of degree three.

\begin{proposition}\label{Gammas}
Let $\Sigma \subset \E$ be a $C^1$ smooth surface with $d(\Sigma) = 3$. If $p \in \Sigma$ is a point of maximum degree, then 
     \begin{equation}\label{mainequality}
\theta^3(\mu_\Sigma, p) = \beta_d \big(A_p \Sigma \big).
    \end{equation}
\end{proposition}
\begin{proof}
By Theorem~\ref{genres}, the assertion \eqref{mainequality} follows by proving the local expansion \eqref{localexp}. 
Following Theorem \ref{t:specialcoord} and its notation, 
our assumptions imply that $\alpha_1= \alpha_2 = 1$, hence
$d_\Sigma(p) = \alpha_1 + 2\alpha_2$ and there exists a $C^1$ smooth parametrization $\Phi : U \rightarrow \E$
of $p^{-1}\Sigma$ around $0$, such that $p^{-1}\Sigma$ has degree 3 at the origin,
\begin{equation}\label{para123}
        \phi(x_1, x_3)=\big(x_1,\phi_2(x),x_3,\phi_4(x) \big)
\end{equation} 
is defined on an open neighborhood $U \subset \R^2$ of $0$ and we have an orthonormal graded basis $(Y_1,Y_2,Y_3,Y_4)$ 
such that we can write $\Phi=F\circ\phi:U\to\Sigma$, where $F(y)=\exp(\sum_{j=1}^4y_jY_j)$.
We remark that $p^{-1}\Sigma$ has degree 3 at the origin and $\Phi(0)=0$.
By \eqref{e:matriceC} restricted to our case, we have
\begin{equation}\label{dinuovomatrix}
D \phi =
\begin{pmatrix}
    1  &  0 \\
    o(1) & \star \\
    0 & 1  \\
    o(1) & o(1)
\end{pmatrix}.
\end{equation}
By \eqref{inducedb}, the induced degrees are $b_1=1$ and $b_2 =2$, hence we introduce the $C^1$ smooth homeomorphism $\eta$
defined in \eqref{funzioneeta} by 
\begin{equation}\label{eta123}
\eta(t_1, t_2) = 
\big(t_1,  {\rm sgn}(t_2) t_2^2/2 \big).
\end{equation}
From \eqref{para123} and \eqref{eta123}, it follows that $\Gamma_1$ and $\Gamma_3$ satisfy  \eqref{localexp}.
On the other hand, thanks to \eqref{dinuovomatrix} and the $C^1$ regularity, we have
\begin{align}\label{stimasuphi2}
 \lvert \phi_2 (x_1, x_3)\rvert &= \Big|\int_0^1 \langle\nabla \phi_2(s x_1, s x_3), (x_1, x_3)\rangle ds\Big|\\
\notag &=\Big|\int_0^1 \langle\nabla \phi_2(s x_1, s x_3)-\nabla\phi_2(0,0), (x_1, x_3)\rangle ds\Big| + \lvert \partial_{x_3} \phi_2(0) x_3  \rvert \\
\notag &\leq \lvert(x_1,x_3)\rvert M(\lvert (x_1,x_3)\rvert) +  \lvert \partial_{x_3} \phi_2(0) x_3  \rvert,
\end{align}
for every $(x_1,x_3)$ in $U$,
where $M(\delta)\coloneqq \max_{x \in \overline{B}_\delta(0)} \lvert \nabla \phi_2(x) - \nabla \phi_2(0)\rvert$. 
Therefore, we obtain
\begin{equation}\label{eq:phi_2(eta)}
 \lvert \Gamma_2(t) \rvert= \lvert \phi_2(\eta(t))\rvert \leq\lvert \eta(t) \rvert M(\lvert \eta (t)\rvert) + \lvert \partial_{x_3}\phi_2(0) \frac{t_2^2}2\rvert=o(t) \quad \mbox{as}\quad t\rightarrow 0.
\end{equation}
It remains to prove that $\Gamma_4$ satisfies the condition arising from \eqref{localexp}, 
namely $\Gamma_4(t)=o(|t|^3)$ as $t\to0$. The proof of this fact will require the assumption on the degree of $\Sigma$. It holds:
\begin{align*}
 \Gamma_4(t)  
&=  \phi_4(\eta(t)) \\
&=\int_0^1 \langle \nabla\phi_4(
(s t_1,  {\rm sgn}(t_2) s t_2^2/2), (t_1, {\rm sgn}(t_2) t_2^2/2) \rangle ds  \notag\\
    &= t_1 \int_0^1 \partial_{x_1} \phi_4(st_1, {\rm sgn}(t_2) s  t_2^2/2) {\rm ds} + {\rm sgn}(t_2)\frac{t_2^2}{2}  \int_0^1 \partial_{x_3} \phi_4(st_1, {\rm sgn}(t_2) s t_2^2/2) {\rm ds}. 
\end{align*}
Our claim follows if we prove that $\partial_{x_1}\phi_4(\eta(t)) = o(t^2)$ and $\partial_{x_3}\phi_4(\eta(t)) = o(t)$ as $t \rightarrow 0$. 
Since $d(\Sigma) < 4$, the coefficient of $\widetilde{Y}_1 \wedge \widetilde{Y}_4$ in \eqref{2vector}
must be equal to zero, obtaining 
\begin{equation}\label{firstv2}
\partial_{x_3} \phi_4 = \frac12(\xi_{13}x_1+\xi_{23}\phi_2) - \Big(\frac{\xi_{23}}{2}x_3 + \frac{1}{6}\xi_{12}(\xi_{23}x_1 \phi_2 + \xi_{13} x_1^2)
\Big)\partial_{x_3} \phi_2.
\end{equation}
By Theorem~\ref{mainstokestheorem}, since $p^{-1}\Sigma$ has maximum degree 3, it must be $\xi_{13}=0$. 
This is a crucial fact for our estimate, since it eliminates the problematic lower order term $x_1$ and turns \eqref{firstv2} into the following
\begin{equation}\label{firstv}
\partial_{x_3} \phi_4 = -\xi_{23}\Big(\frac{x_3}{2} +\frac{x_1}{6}\xi_{12} \phi_2 
     \Big)\partial_{x_3} \phi_2 +\frac12\xi_{23}\phi_2.
\end{equation}
We immediately get $\der_{x_3}\phi_4(\eta(t)=o(t)$, since $\phi_2(\eta(t))=o(t)$ as $t\to0$, due to \eqref{eq:phi_2(eta)}.
To estimate $\der_{x_1}\phi_4$, we exploit the other condition $d(\Sigma) < 5$, hence the vanishing of the coefficient of 
$\Tilde{Y}_3 \wedge \Tilde{Y}_4$ in \eqref{2vector}. We obtain
 \begin{align*}
   \partial_{x_1} \phi_4 
   &= \frac{1}{1- x_1 \frac{\xi_{12}}{2}\partial_{x_3}\phi_2} \Bigg\{\Big(
 \frac{1}{12}\xi_{12}(\xi_{23}x_1 \phi_2 + \xi_{13} x_1^2)
    - \frac{\xi_{23}}{2}x_3\Big) \partial_{x_1}\phi_2  - x_1 \frac{\xi_{12}}{2}\partial_{x_1}\phi_2 \partial_{x_3}\phi_4 \\
    & + \phi_2 \frac{\xi_{12}}{2}\partial_{x_3}\phi_4- \Big(\frac{1}{12}\xi_{12}(\xi_{13}x_1 \phi_2 + \xi_{23}\phi_2^2) + \frac{\xi_{13}}{2}x_3 
 \Big) +\Big( \frac{\xi_{12} \xi_{13}}{4} x_1 x_3 + \frac{\xi_{12} \xi_{23}}{4}\phi_2 x_3\Big) \partial_{x_3}\phi_2\Bigg\}.
\end{align*}
Thus, by Theorem \ref{mainstokestheorem} we get 
 \begin{align*}
   \partial_{x_1} \phi_4
   &= \frac{1}{1- x_1 \frac{\xi_{12}}{2}\partial_{x_3}\phi_2} \Bigg\{\Big(
 \frac{1}{12}\xi_{12}\xi_{23}x_1 \phi_2
    - \frac{\xi_{23}}{2}x_3\Big) \partial_{x_1}\phi_2  - x_1 \frac{\xi_{12}}{2}\partial_{x_1}\phi_2 \partial_{x_3}\phi_4 
    + \phi_2 \frac{\xi_{12}}{2}\partial_{x_3}\phi_4\\ \notag
    & - \frac{1}{12}\xi_{12} \xi_{23}\phi_2^2 + 
    \frac{\xi_{12} \xi_{23}}{4}\phi_2 x_3 \partial_{x_3}\phi_2\Bigg\},
\end{align*}
Since we already know that $\partial_{x_3}\phi_4 (\eta(t)) = o(t)$, $\phi_2(\eta(t))=o(t)$, and $\der_{x_1}\phi_2=o(1)$ 
from \eqref{dinuovomatrix}, the previous formula yields 
$$
\partial_{x_1}\phi_4(\eta(t)) = o(t^2)\quad \mbox{as}\quad t\rightarrow 0,
$$
therefore concluding the proof.
 \end{proof}

\begin{proposition}\label{Gammas2}
    Let $\Sigma \subset \E$ be a  $C^1$ smooth surface with degree $d(\Sigma)=4$. Let $p \in \Sigma$ be a point of maximum degree, then 
     \begin{equation}\label{mainequality2}
        \theta^4(\mu_\Sigma, p) = \beta_d \big(A_p \Sigma \big).
    \end{equation}
\end{proposition}

\begin{proof}
By Theorem \ref{genres}, our claim follows from the local expansion \eqref{localexp}.
In view of Theorem~\ref{t:specialcoord},  our assumptions give $\alpha_1= \alpha_3 = 1$ such that $d_\Sigma(p) = \alpha_1 + 3\alpha_3$ 
and a $C^1$ smooth local parametrization $\Phi : U \rightarrow p^{-1}\Sigma$ around $0\in p^{-1}\Sigma$ such that
$$\phi(x_1, x_4) = (x_1, \phi_2(x), \phi_3(x), x_4)$$
is defined on an open set $U \subset \R^2$ of $0$, and there exists an orthonormal graded basis $(Y_1,Y_2,Y_3, Y_4)$ 
with $\Phi=F\circ\phi$ and $F:\R^4\to\E$ is defined as $F(y)=\exp(\sum_{j=1}^4y_jY_j)$.
We notice that $\Phi(0)=0$ and $p^{-1}\Sigma$ has degree 4 at the origin.
Thus, \eqref{e:matriceC} in our case yields
\begin{equation}\label{formagrado4}
D \phi = 
\begin{pmatrix}
    1  & 0 \\
    o(1) & \star \\
    o(1) & \star \\
    0 & 1 
\end{pmatrix}.
\end{equation}
By \eqref{inducedb} of Theorem \ref{genres}, the induced degrees are $b_1 = 1,b_2 = 3$, hence we introduce the 
$C^1$ smooth homeomorphism $\eta:\R^2\to\R^2$ defined as
$$
\eta(t_1, t_2) = \Big( t_1, \frac{ t_2^3}{3} \Big).
$$
The proof is complete if we verify the conditions of \eqref{localexp} for both $\Gamma_2$ and $\Gamma_3$,
being the conditions for  $\Gamma_1$ and $\Gamma_4$ already satisfied.

Using the form of \eqref{formagrado4} and arguing as in \eqref{stimasuphi2}, we obtain that
\begin{equation}\label{disugsuphij}
    \lvert \phi_j (x_1,x_4) \rvert \leq \lvert(x_1,x_4)\rvert M(\lvert (x_1,x_4)\rvert) +  \lvert \partial_{x_4} \phi_j(0) x_4  \rvert,
\end{equation}
for $j=2,3$ and $M(\delta)\coloneqq \max_{x \in \overline{B}_\delta(0)}\max_{j=2,3} \lvert \nabla \phi_j(x) - \nabla \phi_j(0)\rvert$. 
Then we get 
$$
\lvert \Gamma_2(t) \rvert =\lvert \phi_2 (\eta(t)) \rvert
= o(t)\quad\text{as $t\rightarrow 0$.}
$$
We are finally left to prove that $\Gamma_3(t)=o(t^2)$. Since $d(\Sigma) < 5$, the coefficient of $\widetilde{Y}_3 \wedge \widetilde{Y}_4$ in \eqref{2vector} must be equal to zero,
\[
\partial_{x_1}\phi_3\big(1 - p_1\partial_{x_4}\phi_2\big) + p_1\partial_{x_1}\phi_2 \partial_{x_4}\phi_3 -x_1 \frac{\xi_{12}}{2}\partial_{x_1}\phi_2 + \phi_2 \frac{\xi_{12}}{2} - p_2 \partial_{x_4} \phi_3 +p_3 \partial_{x_4} \phi_2 = 0 ,
\]
where we have set
$$p_i \coloneqq 
\begin{cases}
 \frac{1}{12}\xi_{12}(\xi_{23}x_1 \phi_2 + \xi_{13} x_1^2)
     - \frac{\xi_{23}}{2}\phi_3 \quad i=1, \\
     \\
   \frac{1}{12}\xi_{12}(\xi_{13}x_1 \phi_2 + \xi_{23}\phi_2^2) + \frac{\xi_{13}}{2}\phi_3 
\quad i=2,\\
 \\  \frac{\xi_{12} \xi_{13}}{4} x_1 \phi_3 + \frac{\xi_{12} \xi_{23}}{2}\phi_2 \phi_3 \quad i=3.
\end{cases}
$$
Since $p_1$ vanishes at the origin, we can divide the previous equality by 
$1 - p_1 \partial_{x_4}\phi_2$, hence
\begin{equation}\label{vincologrado4}
	\partial_{x_1} \phi_3 = \frac{1}{1 - p_1\partial_{x_4}\phi_2}\Big( x_1 \frac{\xi_{12}}{2}\partial_{x_1}\phi_2 - \phi_2 \frac{\xi_{12}}{2} -p_1\partial_{x_1}\phi_2\partial_{x_4}\phi_3   - p_3 \partial_{x_4}\phi_2 + p_2\partial_{x_4}\phi_3 \Big).
\end{equation}
Since we already know that \eqref{disugsuphij} gives $\phi_2(\eta(t))=o(t)$, formula \eqref{vincologrado4} immediately gives $\partial_{x_1} \phi_3(\eta(t)) = o(t)$, which leads to
\begin{align}\label{identitygrad4}
   \notag \lvert \Gamma_3(t) \rvert &
    = \lvert \phi_3 (\eta(t))\rvert \\
    &= \left| t_1 \int_0^1 \partial_{x_1} \phi_3(st_1, s t_2^3/3) {\rm ds}  +  t_2^3/3\int_0^1 \partial_{x_4} \phi_3(st_1, s t_2^3/3) {\rm ds} \right| \\
    &\notag= o(t^{2}) \quad \mbox{as}\quad  t\rightarrow 0,
\end{align}
concluding the proof.
\end{proof}

\section{Appendix}\label{appendix}

For the reader's sake, in this appendix we provide a proof of the negligibility theorem needed for our area formula, namely Theorem~\ref{teotrascu}. Such result was stated in \cite[Remark~2]{LeDMag2010} without proof.

We start with a general fact of measure theory, that can be obtained from \cite[2.10.19]{Federer69}.
\begin{lemma}
	Let $X$ be a metric space, let $\mu$ be a Borel measure on $X$ and let $\{E_i\}_{i \in \mathbb{N}}$ be an open covering of $X$ such that $\mu(E_i) < \infty$. Let $Z \subset X$ be a Borel set and suppose that
	$$
	\lim_{r \rightarrow 0^+} r^{-\beta} \mu(\B(p,r)) \geq s > 0
	$$
	whenever $p \in Z$, where $\beta > 0$. Then $\mu(Z) \geq s\, \mathcal{S}^\beta(Z)$.
\end{lemma}

The previous lemma immediately gives the following proposition,
where $\mu_k$ denotes the $k$-dimensional Riemannian surface measure induced on a submanifold $\Sigma\subset\E$ with respect to our fixed left invariant Riemannian metric.

\begin{proposition}\label{propfortras} Let $\Sigma$ be $k$-dimensional $C^{1}$ submanifold of $\E$ and let $\mu_k$ be $k$-dimensional left invariant Riemannian measure. 
If $Z$ is a Borel
set of $\Sigma$ such that for every $p \in Z$
$$ \lim_{r \rightarrow 0^+} r^{-\beta} \mu_k(\B(p,r) \cap \Sigma) = + \infty,
$$
then $\mathcal{S}^\beta(Z) = 0$ for every $\beta>0$.
\end{proposition}
Our objective is to apply the previous proposition. The following lemma gives
a key lower estimate for the surface measure localized by a ball centered
on the submanifold. We will use the following notation
$$
{\rm Box}(0,r) = [-r,r]^2 \times [-r^2, r^2] \times [-r^3, r^3],
$$
that is comparable with the metric ball of radius $r>0$, centered at the origin.
This fact will be used in the proof of the next lemma.

\begin{lemma}\label{represelemma}
There exists $0<\lambda<1$ such that the following statement holds. Let $p \in \Sigma$ be a point of a $C^1$ surface in $\E$. Then for $r>0$ small enough, the formula
    \begin{equation}\label{formul4}
        \mu_2(\Sigma \cap \B(p,r)) = r^{d_\Sigma(p)}\int_{\Tilde{\delta}_{\frac{1}{r}}(\Phi^{-1}(\B(0,r)))} (J \Phi) (\Tilde{\delta}_r u) du,
    \end{equation}
holds, where both the induced dilations $\Tilde{\delta}_r(u) = (r^{d_i} u_i, r^{d_j} u_j)$ and the parametrization
    $\Phi: U \rightarrow \E$ of $p^{-1}\Sigma$ are 
    given by Theorem \ref{t:specialcoord}, and $U \subset \mathbb{R}^2$ is an open neighborhood of the origin.
    Furthermore, we have
    \begin{equation}\label{stimalemma}
        \mu_2(\Sigma \cap \B(p,r)) \geq 
      r^{d_\Sigma(p)} \frac{J \Phi(0)}{2}  \mathcal{L}^2\big(\Tilde{\delta}_{\frac{1}{r}}(\Phi^{-1}((F({\rm Box}(0,\lambda r)))) \big),
    \end{equation}
    for $r > 0$ sufficiently small.
\end{lemma}
\begin{proof}
Up to left translations, it suffices to prove the assertion for the metric open ball $\B(0,r)$ in $\E$.
Indeed, left translations turn out to be isometries with
respect to our fixed left invariant Riemannian metric and
the Riemannian measure $\mu_2$ is defined by the same left invariant Riemannian metric. We have 
\begin{equation}\label{repblow}
\mu_2(\Sigma \cap \B(p,r)) = \mu_2(p^{-1} \Sigma \cap \B(0,r)) = \int_{\Phi^{-1}(\B(0,r))} (J\Phi)(u) du,
\end{equation}
where $J \Phi$ is the Riemannian Jacobian of $\Phi$
with respect to the fixed left invariant Riemannian metric.
By the change of variables $u=\Tilde{\delta}_r(v) = (r^{d_i} v_i, r^{d_j} v_j)$ to the right-hand side of \eqref{repblow} and taking into account that $d_\Sigma(p) = d_i + d_j$, 
we achieve \eqref{formul4}.
Let us consider the change of variables given by $F$, 
that is defined in \eqref{analyticdiffeo}. It is easy to notice that
    \begin{align*}
        F^{-1}(\B(0,r)) &= F^{-1}(\delta_r \B(0,1)) \\
        &= F^{-1} \delta_r F F^{-1}(\B(0,1)) \\
        &= \Tilde{\delta}_r \Tilde{\B}(0,1),
    \end{align*}
where $\Tilde{\B}(0,1)$ is a neighbourhood of the origin in $\mathbb{R}^4$. 
We observe that there exists $0<\lambda<1$ such that
$$
 {\rm Box}(0,\lambda r)\subset \Tilde{\delta}_r {\rm Box}(0,\lambda)\subset \Tilde{\delta}_r \Tilde{\B}(0,1) \subset \Tilde{\delta}_r {\rm Box}(0,1/\lambda) = {\rm Box}(0,r/\lambda )
$$
and therefore we get 
\begin{equation}\label{ballexp}
F({\rm Box}(0, r\lambda )) \subset \B(0,r) \subset F({\rm Box}(0, r/\lambda )).
\end{equation} 
Thus, \eqref{stimalemma} is a direct consequence of \eqref{formul4} and of the first inclusion of \eqref{ballexp}.
\end{proof}

\begin{proposition}\label{grado41}
Let $\alpha >1/2$ and consider a $2$-dimensional $C^{1,\alpha}$ smooth submanifold $\Sigma\subset\E$ such that $d(\Sigma) \geq 4$ and $d_{\Sigma}(p) = 3$, for some $p\in\Sigma$. Then we have
$$ \lim_{r \rightarrow 0} \frac{\mu_2(\Sigma \cap \B(p,r))}{r^{d(\Sigma)}} = + \infty.$$
\end{proposition}
\begin{proof}
Using Theorem~\ref{t:specialcoord} and its notation, by our assumptions we must have $\alpha_1=\alpha_2=1$ such that $d_\Sigma(p) = \alpha_1 + 2\alpha_2$ and $\Phi : U \rightarrow \E$ is a $C^1$ smooth parametrization of 
    $p^{-1}\Sigma$, where  
\begin{equation*}
        \phi(x_1, x_3) = (x_1, \phi_2(x), x_3, \phi_4(x))
\end{equation*} 
is defined on an open neighborhood $U\subset\R^2$ of $0$. We have also used the change of variables $F:\R^4\to\E$ of Theorem~\ref{t:specialcoord}, such
that $\Phi=F\circ\phi$. In particular, by \eqref{e:matriceC} in our case, we infer that
$$
D \phi = 
\begin{pmatrix}
    1   & 0 \\
    o(1) & \star \\
    0 & 1 \\
    o(1) & o(1) 
\end{pmatrix}.
$$
Thus, there exists $C_4 > 0$ such that 
$\lvert \phi_4(u) \rvert \leq C_4 \lvert u \rvert^{1 + \alpha}$ for every $|u|$
sufficiently small. 
Let $0 < \lambda < 1$ be as in Lemma \ref{represelemma} and let us consider the rescaled set $\Tilde{\delta}_{\frac{1}{r}}\big(\phi^{-1}({\rm Box}(0,\lambda r))\big)$, that is the following
$$
\Biggl\{ (x_1, x_3) \ : \ \frac{\lvert x_1 \rvert}{\lambda} \leq 1, \; \frac{\lvert x_3 \rvert}{\lambda^2} \leq 1, \; \frac{\lvert \phi_2(r x_1, r^2 x_3) \rvert }{\lambda r} \leq 1, \; \frac{\lvert \phi_4(r x_1, r^2 x_3) \rvert}{(\lambda r)^3} \leq 1  \Biggr\}.
$$
Moreover, from \eqref{e:matriceC}, we obtain 
\begin{align}\label{disuguagrado4}
\notag    \lvert \phi_2( x_1, x_3) \rvert 
    &=\left| \int_0^1 \langle \nabla \phi_2( t x_1 ,  t x_3 ), (x_1 , x_3) \rangle {\rm dt}\right| \\
 \notag   &\le \left|\int_0^1 \langle \nabla \phi_2( t x_1 ,  t x_3 ) - \nabla \phi_2 (0 , 0) , (x_1, x_3)\rangle {\rm dt}\right| + \left|\int_0^1 \langle \nabla \phi_2(0,0 ), (x_1 , x_3) \rangle {\rm dt} \right| 
    \\
    & \leq C \lvert (x_1, x_3)\rvert^{1+\alpha} +  \lvert \partial_{x_3} \phi_2(0) x_3  \rvert.
\end{align}
Thus, for every $(x_1, x_3)$ in any fixed compact set and $r>0$ sufficiently small, we infer that
\begin{equation}
 C r^\alpha \lvert ( x_1, r x_3)\rvert^{1+\alpha} + \lvert \partial_{x_3} \phi_2(0)\rvert r \lambda^2 \le \lambda,
\end{equation}
establishing the estimate
$\lvert \phi_2(r x_1, r^2 x_3) \rvert\le 
\lambda r$.
Thanks to our estimates on $\phi_4$ and $\phi_2$, one can verify that $\Tilde{\delta}_{\frac{1}{r}}\big(\phi^{-1}({\rm Box}(0,\lambda r))\big)$ contains the subset
$$ E_r = \Bigl\{ (x_1, x_3) \ : \ \lvert x_1 \rvert \leq \lambda, \quad \lvert x_3 \rvert \leq \lambda^2, 
\quad \lvert (x_1, r x_3) \rvert \leq \widetilde{C}_4 r^{\frac{2 - \alpha}{1 + \alpha}} 
\lambda^{\frac{4}{1 + \alpha}}
\Bigr\}.
$$
We observe that in Lemma \ref{represelemma} it is possible to assume that 
$0<\lambda<1$ is sufficiently small, such that $C_4\widetilde{C}_4^{1+\alpha}\lambda<1$.
Using the change of variable $x'_3 = r x_3$, we obtain 
$$
\mathcal{L}^2(E_r) = \frac{1}{r} \mathcal{L}^2 \Big( \Bigl\{ (x_1, x'_3) \ : \ \lvert x_1 \rvert \leq \lambda, \; \lvert x'_3 \rvert \leq r \lambda^2, 
\; \lvert (x_1, x'_3) \rvert \leq C_4 r^{\frac{2 - \alpha}{1 + \alpha}} \lambda^{\frac{4}{1 + \alpha}}\Bigr\}\Big).$$
Therefore, for $r > 0$ sufficiently small 
we get that
\begin{align*}
    \mathcal{L}^2(E_r) &\geq \frac{1}{r} \mathcal{L}^2 \Bigg( \Biggl\{ (x_1, x_3) \ : \ \lvert x_3 \rvert \leq r \lambda^2, \; \max \{ \lvert x_1 \rvert, \lvert x_3 \rvert \} \leq \frac{1}{2} C_4 r^{\frac{2 - \alpha}{1 + \alpha}} \lambda^{\frac{4}{1 + \alpha}}\Biggr\}\Bigg) \\
    &= \frac{1}{r}\mathcal{L}^2 \Big( \Bigl\{ (x_1, x_3) \ : \ \lvert x_3 \rvert \leq r \lambda^2, \; \lvert x_1 \rvert \leq 
    \frac{1}{2} C_4 r^{\frac{2 - \alpha}{1 + \alpha}} \lambda^{\frac{4}{1 + \alpha}}\Bigr\}\Big)\\
    &= 2 C_4 \lambda^{\frac{6 + 2 \alpha}{1 + \alpha}} r^{\frac{2 - \alpha}{1 + \alpha}},
\end{align*}
using the assumption $\alpha > 1/2$.
Combining \eqref{stimalemma} and the previous estimate, we get
\begin{align*}
\frac{\mu_2(\Sigma \cap \B(p,r))}{r^{d(\Sigma)}} &\geq C_4 \lambda^{\frac{6 + 2 \alpha}{1 + \alpha}} \frac{r^{\frac{2 - \alpha}{1 + \alpha}}}{r^{d(\Sigma) - 3}} J \Phi(0) \\
&\geq C_4\lambda^{\frac{6 + 2 \alpha}{1 + \alpha}} J \Phi(0)
r^{\frac{2 - \alpha}{1 + \alpha} - 1} \longrightarrow + \infty  \quad \mbox{as } r \rightarrow 0,
\end{align*}
since $\alpha > \frac12$. This concludes the proof.
\end{proof}

\begin{proposition}\label{grado42}
Let $\alpha >1/2$ and consider a $2$-dimensional $C^{1,\alpha}$ smooth submanifold $\Sigma\subset\E$, such that $d(\Sigma) \geq 4$ and $d_{\Sigma}(p) = 2$
for some $p\in\Sigma$.
Then we have
$$ \lim_{r \rightarrow 0} \frac{\mu_2(\Sigma \cap\B(p,r))}{r^{d(\Sigma)}} = + \infty.$$   
\end{proposition}
\begin{proof}
  Applying Theorem \ref{t:specialcoord}, and its notation, 
    by our assumptions we must have $\alpha_1=2$ such that $d_\Sigma(p) = \alpha_1$ and $\Phi : U \rightarrow \E$ is a $C^1$ smooth parametrization of 
    $p^{-1}\Sigma$, where  
\begin{equation*}
        \phi(x_1, x_2) = (x_1, x_2, \phi_3(x), \phi_4(x))
\end{equation*} 
is defined on an open neighborhood $U \subset \R^2$ of $0$.  We have also used the change of variables $F:\R^4\to\E$ of Theorem~\ref{t:specialcoord}, such
that $\Phi=F\circ\phi$. In particular, by \eqref{e:matriceC} in our case, we get
\begin{equation*}
D \phi = 
\begin{pmatrix}
    1 & 0 \\
    0 & 1 \\
    o(1) & o(1) \\
    o(1) & o(1) 
\end{pmatrix}.
\end{equation*}
Then, there exist positive real constant $c_3, c_4$  such that 
\begin{equation}\label{stime24}
\lvert \phi_3(u) \rvert \leq c_3 \lvert u \rvert^{1+ \alpha} \quad  \mbox{and} \quad \lvert \phi_4(u) \rvert \leq c_4 \lvert u \rvert^{1 + \alpha} \quad \mbox{for } |u| \;\mbox{small}.
\end{equation}
Let $0 < \lambda < 1$ as in Lemma \ref{represelemma}, then we set $\Tilde{\delta}_{\frac{1}{r}}\big(\phi^{-1}({\rm Box}(0,\lambda r))\big)$, that can be written as 
$$
 \Biggl\{ (x_1, x_2) \ : \ \frac{\lvert x_1 \rvert}{\lambda} \leq 1, \; \frac{\lvert x_2 \rvert}{\lambda} \leq 1, \; \frac{\lvert \phi_3(r x_1, r x_2) \rvert }{(\lambda r)^2} \leq 1, \; \frac{\lvert \phi_4(r x_1, r x_2) \rvert}{(\lambda r)^3} \leq 1  \Biggr\}.
$$
We observe that in Lemma \ref{represelemma} we may choose a possibly smaller $0<\lambda<1$, such that $c_4\widetilde{c}_4^{1+\alpha}\lambda<1$ and $c_3\widetilde{c}_3^{1+\alpha}\lambda<1$. Thus, applying the estimates \eqref{stime24}, 
the latter set contains 
\begin{align*}
&\Bigl\{ (x_1, x_2) \ : \ \lvert x_1 \rvert \leq \lambda, \; \lvert x_2 \rvert \leq \lambda, \; \lvert (x_1, x_2) \rvert \leq \widetilde{c}_3\lambda^{\frac{3}{1 + \alpha}} r^{\frac{1-\alpha}{1 + \alpha}}, \; \lvert (x_1, x_2) \rvert \leq \widetilde{c}_4 \lambda^{\frac{4}{1 + \alpha}} r^{\frac{2 - \alpha}{1 + \alpha}}  \Bigr\} \\
&\supseteq \Bigl\{ (x_1, x_2)  :   
\max \{\lvert x_1\rvert, \lvert x_2 \rvert \}
\leq \frac12\Tilde{c}_4\lambda^{\frac{4}{1 + \alpha}} r^{\frac{2 - \alpha}{1 + \alpha}} \Bigr\}= A_r.
\end{align*}
Clearly, we have $\mathcal{L}^2(A_r) =  \widetilde{c}_4^2 \lambda^{\frac{8}{1 + \alpha}} r^{\frac{4 - 2 \alpha}{1 + \alpha}}$.
Using \eqref{stimalemma}, we obtain 
\begin{align*}
\frac{\mu_2(\Sigma \cap \B(p,r))}{r^{d(\Sigma)}} &\geq  \widetilde{c}_4^2 
\lambda^{\frac{8}{1 + \alpha}} \frac{r^{\frac{4 - 2 \alpha}{1 + \alpha}}}{r^{4 - d_\Sigma(p)}} \frac{J \Phi(0)}{2} \\
&= \widetilde{c}_4^2
\lambda^ {\frac{8}{1 + \alpha}} \frac{J \Phi(0)}{2} r^{\frac{4 - 2\alpha}{1 + \alpha} - 2} \longrightarrow + \infty \quad \mbox{as } r \rightarrow 0,
\end{align*}
whenever $\alpha > \frac12$, concluding the proof. 
\end{proof}

\begin{proposition}\label{grado31}
Let $\alpha >1/3$ and consider a $2$-dimensional $C^{1,\alpha}$ smooth submanifold $\Sigma\subset\E$, such that $d(\Sigma) =3$ and $d_{\Sigma}(p) = 2$
for some $p\in\Sigma$.
Then we have
$$ \lim_{r \rightarrow 0} \frac{\mu_2(\Sigma \cap \B(p,r))}{r^{d(\Sigma)}} = + \infty.$$    
\end{proposition}
\begin{proof}
  Applying Theorem \ref{t:specialcoord}, and its notation, 
    by our assumptions we must have $\alpha_1=2$ such that $d_\Sigma(p) = \alpha_1$ and $\Phi : U \rightarrow \E$ is a $C^1$ smooth parametrization of 
    $p^{-1}\Sigma$, where  
\begin{equation*}
        \phi(x_1, x_2) = (x_1, x_2, \phi_3(x), \phi_4(x))
\end{equation*} 
is defined on an open neighborhood $U \subset \R^2$ of $0$. Moreover, there exists an orthonormal graded basis $(Y_1, \ldots, Y_4)$ 
with change of variables $F:\R^4\to\E$ of Theorem~\ref{t:specialcoord}, such
that $\Phi=F\circ\phi$. In particular, by \eqref{e:matriceC} in our case, we have
$$
D \phi
=
\begin{pmatrix}
    1  & 0 \\
    0 & 1 \\
    o(1) & o(1) \\
    o(1) & o(1) \\
\end{pmatrix}.
$$
Thus, there exists $c_3 > 0$ such that $|\phi_3(u)| \leq c_3 |u|^{1+\alpha}$ for $|u|$ small. The same holds for $\phi_4$ but we notice that it does not suffice to obtain our claim. We have to exploit the assumption $d(\Sigma) < 4$: in other words 
the coefficient of $\widetilde Y_1 \wedge \widetilde Y_4$ must be equal to zero in \eqref{2vector}, equivalently 
\begin{equation}\label{secondvv}
\partial_{x_2} \phi_4 = -\frac{\xi_{23}}{2}\phi_3 + \frac{1}{2}  (\xi_{13}x_1 + \xi_{23}x_2) 
\partial_{x_2} \phi_3  - \frac{1}{6}\xi_{12}(\xi_{23}x_1 x_2 +
\xi_{13} x_1^2).
\end{equation}
The same holds for the coefficient of $\widetilde Y_2 \wedge \widetilde Y_4$, hence 
 locally in $U$ we have
\begin{equation}\label{firstvv}
\partial_{x_1} \phi_4 =-\frac{\xi_{13}}{2}\phi_3 + \frac{1}{2}  (\xi_{13}x_1 + \xi_{23}x_2) 
\partial_{x_1} \phi_3 
 + \frac16 \xi_{12} (\xi_{13} x_1 x_2 + 
 \xi_{23} x_2^2).
\end{equation}
 By the form of \eqref{e:matriceC} and the identities \eqref{firstvv} and \eqref{secondvv} we get that
$$\lvert \nabla \phi_4(x_1,x_2) \rvert = 
O(\lvert (x_1, x_2) \rvert^{1+\alpha}) \quad \mbox{in }  U.$$ 
This implies that $\lvert \phi_4(u) \rvert \leq c_4 \lvert u \rvert^{2+\alpha}$ for $\lvert u \rvert$ small enough and $c_4>0$.
Let $0 < \lambda < 1$ be as in Lemma \ref{represelemma} and let us consider
the set $\Tilde{\delta}_{\frac{1}{r}}\big(\phi^{-1}({\rm Box}(0,\lambda r))\big)$, that is equal to
\begin{eqnarray*}
\Biggl\{ (x_1, x_2) \ : \ \frac{\lvert x_1 \rvert}{\lambda} \leq 1, \; \frac{\lvert x_2 \rvert}{\lambda} \leq 1, \; \frac{\lvert \phi_3(r x_1, r x_2) \rvert }{(\lambda r)^2} \leq 1, \; \frac{\lvert \phi_4(r x_1, r x_2) \rvert}{(\lambda r)^3} \leq 1  \Biggr\}.
\end{eqnarray*}
Observing that Lemma \ref{represelemma} holds also for a smaller $0<\lambda<1$, we apply our estimates to $\phi_3$ and $\phi_4$ assuming that $c_4\widetilde{c}_4^{2+\alpha}\lambda<1$ and $c_3\widetilde{c}_3^{1+\alpha}\lambda<1$.
As a result, the previous set contains
\begin{align*}
&\Biggl\{ ( x_1, x_2) :  \lvert x_1 \rvert \leq \lambda, \; \lvert x_2 \rvert \leq \lambda, \; \lvert (x_1, x_2) \rvert \leq \widetilde{c}_3\lambda^{\frac{3}{1 + \alpha}} r^{\frac{1-\alpha}{1 + \alpha}}, \; \lvert (x_1, x_2) \rvert \leq \widetilde{c}_4\lambda^{\frac{4}{2+\alpha}} r^{\frac{1-\alpha}{2+\alpha}} \Biggr\} \\ 
& \supseteq \Bigl\{ (x_1, x_2)  :  \max \{\lvert x_1\rvert, \lvert x_2 \rvert \}
\leq \frac12\Tilde{c}_4\lambda^{\frac{3}{1 + \alpha}} r^{\frac{1 - \alpha}{1 + \alpha}} \Bigr\} = E_r.
\end{align*}
Then, combining \eqref{stimalemma} with $\mathcal{L}^2(E_r) = \widetilde{c}_4^2 \lambda^{\frac{6}{1+\alpha}} r^{\frac{2 - 2\alpha}{1 + \alpha}}$, 
we get
\begin{align*}
    \frac{\mu_2(\Sigma \cap \B(p,r))}{r^{d(\Sigma)}} &\geq \widetilde{c}_4^2 \lambda^{\frac{6}{1+\alpha}}\frac{ r^{\frac{2 - 2\alpha}{1 + \alpha}}}{r^{d(\Sigma) - d_\Sigma(p)}}  \frac{J \Phi(0)}{2}\\
&=\widetilde{c}_4^2 
\lambda^ {\frac{6}{1 + \alpha}} \frac{J \Phi(0)}{2} r^{\frac{2-2\alpha}{1 + \alpha} - 1} \longrightarrow + \infty \quad \mbox{as } r \rightarrow 0,
\end{align*}
if 
$\alpha > \frac13$. This concludes the proof.
\end{proof}

\begin{theorem}\label{teotrascu}
	Let $\Sigma\subset\E$ be a $2$-dimensional $C^{1,\alpha}$ smooth submanifold with $3\le d(\Sigma)\le 4$ and let $C(\Sigma)\subset \Sigma$ be the subset of points of degree less than $d(\Sigma)$. Then we have
	\begin{equation}\label{eq:SN0}
		\mathcal{S}^{d(\Sigma)}(C(\Sigma)) = 0
	\end{equation}
	whenever $(d(\Sigma)-2)/d(\Sigma)<\alpha\le 1$. 
\end{theorem}
\begin{proof}
	If $d(\Sigma) = 4$, then both Proposition \ref{grado41} and Proposition \ref{grado42} 
	imply that
	$$ \lim_{r \rightarrow 0} \frac{\mu_2(\Sigma \cap \B(p,r))}{r^4} = + \infty$$
	whenever $p\in\Sigma$ and $d_\Sigma(p)<4$. 
	Thus, Proposition~\ref{propfortras} immediately gives \eqref{eq:SN0}.
	In the case $d(\Sigma)=3$, Proposition \ref{grado31} yields
	$$
	 \lim_{r \rightarrow 0} \frac{\mu_2(\Sigma \cap \B(p,r))}{r^3} = + \infty
	$$
	whenever $d_\Sigma(p)<3$ and again \eqref{eq:SN0} holds, 
	due to Proposition \ref{propfortras}.
\end{proof}

{\bf Acknowledgements.} The authors wish to thank Francesca Tripaldi for inspiring conversations on the algebraic properties of the left invariant forms of the Engel group.

\bibliographystyle{abbrv} 
\bibliography{References}
\end{document}